\documentclass[a4paper,10pt]{article}
\usepackage{amsfonts}
\usepackage{enumitem}
\usepackage{amsthm}
\usepackage{nccmath}
\usepackage{amsmath}
\usepackage{amssymb}
\usepackage{setspace}
\advance\textwidth by 50pt
\advance\oddsidemargin by -25pt
\advance\evensidemargin by -25pt

\newtheorem{lemma}{Lemma}[section]
\newtheorem{proposition}[lemma]{Proposition}
\newtheorem{theorem}[lemma]{Theorem}
\newtheorem{definition}{Definition}[section]
\newtheorem{principle}{Principle}[section]
\newtheorem{pbm}{Problem}[section]
\newtheorem{corollary}{Corollary}[section]
\theoremstyle{definition}
\newtheorem{remark}[lemma]{Remark}

\newcommand{\be}{\begin{equation}}
\newcommand{\ee}{\end{equation}}
\newcommand{\pfreem}{p_{\textrm{\tiny{s}}}}

\newcommand{\tildh}{\tilde{h}}

\newcommand{\Freem}{\mathcal{H}}
\newcommand{\pFreem}{\mathcal{H}\p}
\newcommand{\optfreemin}{\sigma_{\bar h}}
\newcommand{\pfreemin}{\overline{p}_\textrm{\tiny{s}}}
\newcommand{\freemin}{\overline{h}}
\newcommand{\freem}{h}
\newcommand{\nut}{\tilde{\nu}}
\newcommand{\R}{\mathbb{R}}
\newcommand{\optFreem}{\mathcal{P}_{ac}(\mathbb{R}^3)}
\newcommand{\optfreem}{\sigma}

\newcommand{\pr}{p_\textrm{\scriptsize{ref}}}

\newcommand{\ps}{p_\textrm{\tiny{s}}}
\newcommand{\s}{_\textrm{\tiny{s}}}
\newcommand{\p}{_\textrm{\tiny{p}}}

\newcommand{\renumerate}[1]{\begin{enumerate}[label=(\textit{\roman{*}}), ref=\textit{(\roman{*})}] #1 \end{enumerate}}
\newcommand{\push}{\textrm{\#}}

\newcommand{\bF}{\mathbf F}
\newcommand{\bx}{\mathbf x}
\newcommand{\by}{\mathbf y}
\newcommand{\bZ}{\mathbf Z}
\renewcommand{\geq}{\geqslant}                                     
\renewcommand{\leq}{\leqslant}                                     
\newcommand{\fc}{f_\textrm{\scriptsize{cor}}} 
\renewcommand{\gg}{g_\textrm{\scriptsize{grav}}} 

\begin{document}

\vspace{3mm}
\begin{center}
\Large
{\bf Free upper boundary value problems for the semi-geostrophic equations}

\vspace{3mm}

\large
M.J.P. Cullen$^{*}$, D.K. Gilbert, T. Kuna  and B. Pelloni

\vspace{5mm}
{\em 
Department of Mathematics and Statistics,

University of Reading, 

Reading RG6 6AX, UK}

\vspace{2mm}
$^{*}$ also {\em Met Office, Fitzroy Road, 

Exeter EX1 3PB, UK}

\vspace{3mm}
\today
\end{center}
\normalsize

\begin{abstract}
The  semi-geostrophic system 
 is widely used in the modelling of large-scale atmospheric flows. In this paper, we prove existence of solutions of the incompressible semi-geostrophic equations in a fully three-dimensional domain with a free upper boundary condition.  
 The main structure of the proof follows the pioneering work of Benamou and Brenier  \cite{benamou}, who analysed the same system but with a rigid boundary condition. However, there are very significant new elements required in our proof of the existence of solutions  for the incompressible free boundary problem. The proof uses on optimal transport results as well as the analysis of  Hamiltonian ODEs in spaces of probability measures given by Ambrosio and Gangbo \cite{ambgangbo}. We also show how these techniques can be modified to yield the analogous result for the compressible version of the system.
\end{abstract}

\section{Introduction}\label{introduction}
The fully compressible semi-geostrophic system, posed in a domain of the form $[0,\tau)\times\Omega$, with $\Omega\subset \R^3$ a bounded subset of  the physical space,
is the following system of equations:
\begin{eqnarray}
&&\label{commom}
D_t \mathbf u^g + f_\textrm{\scriptsize{cor}} \mathbf e _3 \times \mathbf u + \nabla \phi + \frac{1}{\rho } \nabla p = 0,\\
&&\label{comad}
D_t \theta = 0,\\
&&\label{comcont}
D_t \frac{1}{\rho } = \frac{1}{\rho }\nabla \cdot \mathbf u,\\
&&\label{comgeo}
f_\textrm{\scriptsize{cor}} \mathbf e _3 \times \mathbf u^g + \nabla \phi + \frac{1}{\rho } \nabla p = 0,\\
&&\label{comstate}
p = R\rho \theta \left (\frac{p}{p_\textrm{\scriptsize{ref}}}\right )^{\frac{\kappa -1}{\kappa }},
\end{eqnarray}
where $D_t$ denotes the lagrangian derivative operator:
 \be
D_t=\partial_t+\mathbf u\cdot\nabla
\label{lagrder}\ee

The unknowns in the above equations are $\mathbf u^g = (u^g_1, u^g_2, 0)$, $\mathbf u = (u_1, u_2, u_3)$, $p$, $\rho $, $\theta$; we assume $R$, $\fc$ and $p_\textrm{\scriptsize{ref}}$ constant, and indeed we will assume $\fc=1$ in what follows.  We also assume $\Phi(\bx)=g_{grav} x_3$. The physical significance of each variable is given in the Appendix.   

\smallskip
This system is obtained as an approximation to the laws of thermodynamics and to the compressible Navier-Stokes equations, the fundamental equations that describe the behaviour of the atmosphere, or more precisely the version obtained when viscosity is neglected, known as the Euler equations. The particular approximation made in the derivation of the semi-geostrophic system is  valid on scales where the effects of rotation dominate the flow. In this case, the effect of the Coriolis and of the pressure gradient force are  balanced, and equation (\ref{comgeo}) is precisely a formulation of hydrostatic and geostrophic balance.  
The remaining equations formulate other physical properties:  (\ref{commom}) is the momentum equation; (\ref{comad}) represents the adiabatic assumption; (\ref{comcont}) is the continuity equation and (\ref{comstate}) is the equation of state which relates the thermodynamic quantities to each other.

The semi-geostrophic system was first introduced by Eliassen \cite{eliassen} and then rediscovered by Hoskins \cite{hoskins}. It admits more singular behaviour in the solutions than other reductions with a simpler mathematical structure, such as the quasi-geostrophic system, and for this reason this system been used in particular to describe the formation of atmospheric fronts. 

\smallskip
For an accurate representation of  the behaviour of large-scale atmospheric flow, one should consider the fully compressible semi-geostrophic equations with variable Coriolis parameter and a free upper boundary condition. The complexity of this problem means that so far results have only been obtained after relaxing one or more of these conditions. We give  a brief summary of these results.

In \cite{benamou}, Benamou and Brenier assumed the fluid to be  incompressible, the Coriolis parameter  constant and the boundaries rigid. The problem  they considered, written in dimensionless scalar form, is posed in a fixed domain $\Omega\subset \R^3$ and  given by
\be
\left\{\begin{array}{ll}
D_tu_1^g -  u_2 + \frac{\partial p}{\partial x_1} = 0,&\\ 
D_tu_2^g +  u_1 +\frac{\partial p}{\partial x_2} = 0,&\\
D_t\rho = 0,& (t,x)\in [0, \tau ) \times \Omega\\ 
\nabla \cdot \mathbf u = 0,&\\ 
\frac{\partial p}{\partial x_1}=u_2^g ,\quad \frac{\partial p}{\partial x_2}=-u_1^g,\quad \frac{\partial p}{\partial x_3} = -\rho.&
\end{array}\right.
\label{sginc}
\ee
The equations are to be solved subject  to  appropriate initial conditions, and the rigid boundary conditions
\be\label{bbbound}
\mathbf u \cdot \mathbf n = 0 \qquad\qquad \qquad  \qquad\qquad  \qquad (t,x)\in [0, \tau ) \times \partial \Omega ,
\ee
where 
$\partial \Omega $ represents the boundary of $\Omega $ and $\mathbf n$ is the outward unit normal to $\partial \Omega $.  

Using a change of variables, first introduced by Hoskins in \cite{hoskins}, one derives the so-called {\em dual formulation} of the system, that elucidates the Hamiltonian structure of the problem. Indeed, in this formulation,  the equations are interpreted as a Monge-Amp\`{e}re equation coupled with a transport problem, and this elegant interpretation yields the proof of the existence of  weak solutions of the system in dual space, based on the groundbreaking work of Brenier \cite{brenier}. 

This result was generalised in \cite{maroofi} to prove existence of weak solutions for the 3-dimensional compressible system (\ref{commom})-(\ref{comstate}),   still assuming a fixed boundary and a rigid boundary condition.  

In \cite{gangbo}, Cullen and Gangbo relaxed the assumption of rigid boundaries assuming a more physically appropriate free boundary condition.  However, they made the  additional assumption of a constant potential temperature, and thus obtained a 2-D system, known as the semi-geostrophic shallow water system, posed on a fixed two-dimensional domain. After passing to dual variables, they showed existence of weak solutions of the resulting dual problem.

The above results were obtained for the dual space formulation of the equations, which is the setting we also consider in the present paper. However, we mention for completeness more recent results regarding  the existence of solutions in the original physical variables. The first step in this direction was taken by Cullen and Feldman, who proved in \cite{feldman}  the existence of Lagrangian solutions in physical variables, a result that was extended in  \cite{CGP1} to the compressible system.  Recently, Ambrosio et al have succedeed in proving  existence of solutions for the  Eulerian formulation, in cases when there are no boundary effects  \cite{ambnew,ambnew2}.

\smallskip
In this paper, we extend the results above to prove the existence of dual-space solutions for the {\em incompressible system, in three-dimensional space, in a domain with a free upper boundary}. This result is stated in Theorem \ref{new5.5} , and is a direct but substantial extension of the results of Cullen and Gangbo. The proof differs from the one given in \cite{gangbo} also in its use of the approach introduced in \cite{CGP1}, namely it exploits the general theory of  Hamiltonian ODEs  in spaces of probability measures given in \cite{ambgangbo}. The strategy of the proof is to show that the Hamiltonian of the system, given by the dual energy, satisfies the necessary conditions to invoke the general theory of \cite{ambgangbo}, and that its superdifferential coincides precisely with the dual velocity of the flow. This, coupled with the existence of the optimal transport map for the given cost function,  yields the desired result.
We also sketch the extension of this proof to the compressible case. Namely, by writing the equations in pressure coordinates, we extend the result of  \cite{maroofi}, who considered the compressible equations but assumed rigid boundary conditions,  to the more physically relevant case of free boundary conditions. 
 
We mention that recently Caffarelli and McCann \cite{caffarelli2010free} have developed extensively a general theory of optimal transport in domains with free boundaries.  It would be interesting to verify whether these general results can be used to give an alternative proof of the problem considered here.

\smallskip
The paper is organised as follows:

In Section 2, we summarise the results of Benamou and Brenier on the solution of the incompressible 3-D system in dual space, with rigid boundary conditions. The proof of this result sets the strategy for all generalisations, and we highlight how our approach differs from this.

In Section 3, we consider the same problem but assume a more realistic free boundary condition on the top boundary (the surface of the fluid). We first summarise the results for the 2-D case obtained by Cullen and Gangbo, then give the proof for the 3-D case.  This is the main result of this paper. 

In Section 4, we extend the results to the compressible system.  In view of the fact that, in pressure coordinates, the two problems are formally identical, this extension does not introduce any new element.

In the Appendix, we list various definitions and the notation we use throughout, as well as some general results in the theory of optimal transport and Hamiltonian flows that we appeal to in the proof of our results.

\section{The incompressible semi-geostrophic system in a fixed domain}
\setcounter{equation}{0}

We start by  describing the strategy common to proving the existence of solutions, in a particular set of coordinates, in all cases we examine. The original approach is due to  Benamou and Brenier \cite{benamou}. 

Let $\Omega\subset \R^3$ be a fixed bounded domain, and $\tau>0$ a fixed constant.  Consider the system of equations  (\ref{sginc}), with suitable prescribed initial conditions and the rigid boundary conditions given by (\ref{bbbound}).
%

\medskip 
The geostrophic energy,  which is conserved by the flow, is given by 
\begin{equation}\label{bbenergy}
E = \int_\Omega \! \bigg( \frac{1}{2}((u_1^g)^2 + (u_2^g)^2) + \rho x_3 \bigg)\, d\mathbf x.
\end{equation}

An important physical property of the flow described by the semigeostrophic approximation  is summarised in the following fundamental principle. 

\begin{principle}[Cullen's stability principle]
Stable solutions of (\ref{sginc})-(\ref{bbbound}) correspond to solutions that, at each fixed time $t$,  minimise the energy $E$ given by (\ref{bbenergy}) with respect to the rearrangements of particles, in physical space, that conserve the absolute momentum $(u_1^g-x_2, u^g_2+x_1)$ and the density $\rho$.
\label{princ1}\end{principle}
This was expressed in  \cite{shutts} as the requirement that states corresponding to critical points of (\ref{bbenergy}) with respect to such rearrangements of particles in physical space are states in hydrostatic and geostrophic balance. The evolution of states that are critical points of the energy but not minima cannot be described by the semi-geostrophic approximation \cite{cullen}.

The significance of Brenier's work is in the elucidation of the precise mathematical meaning of this minimisation principle, and its mathematical formulation in the framework of convex analysis and optimal transport theory. This machinery can be used after a change of variables, introduced by Hoskins \cite{hoskins} and motivated by physical considerations. In these variables,  the problem is formulated mathematically in Hamiltonian form, and the  time evolution of the velocity is expressed explicitly. 

\subsubsection*{Formulation in dual variables}
The  change to dual coordinates $\mathbf y = \mathbf T(t, \mathbf x)$ is defined by
\begin{equation}\label{bbgeovariables}
 \mathbf T: \Omega\to \R^3: \qquad T_1(\bx) = x_1 + u_2^g, \quad T_2(\bx) = x_2 - u_1^g, \quad T_3(\bx) = -\rho .
\end{equation}
Note that (\ref{sginc}) implies 
\[ (y_1 - x_1, y_2 - x_2, y_3) = \nabla p .\]
The energy functional (\ref{bbenergy}) is formulated in dual variables as 
\begin{equation}\label{bbenergyrewrite}
E(t,\bx,\mathbf T) = \int_\Omega \! \bigg( \frac{1}{2}\{ |x_1 - {T}_1(\mathbf x)|^2 + |x_2 - {T}_2(\mathbf x)|^2 \} - x_3{T}_3(\mathbf x) \bigg) \, d\mathbf x.
\end{equation}
The geostrophic coordinates are related to Cullen's stability principle through  the so-called {\em geopotential} $P(t, \mathbf x)$, defined as
\begin{equation}\label{bbP}
P(t, \mathbf x) = \frac{1}{2}(x_1^2 + x_2^2) + p(t, \mathbf x).
\end{equation}
One can perform a formal variational computation, with respect to variations $\varphi$ of particle position  satisfying the incompressibility constraint  $\nabla\cdot\varphi=0$ and that conserve absolute momentum so that $u_1^g-\varphi_2=u_2^g+\varphi_1=0$. This computation 
indicates that, for the energy   in (\ref{bbenergyrewrite}) to be stationary, it must hold that $\mathbf T(\bx) = \nabla P$, and  that  the condition for the energy to be minimised is that $D^2P$ is positive definite, where $D^2$ is the Hessian.  Positive definiteness of $D^2P$ implies that $P$ is convex, see \cite{cullen, purser, dkgthesis, shutts}.   Hence the stability principle can be formulated as a convexity principle.

\begin{principle}[Cullen's convexity principle]
Minima of the energy (\ref{bbenergy}), with respect to variations as in Principle \ref{princ1},  correspond to a  geopotential $P(t,\bx)$, as given by (\ref{bbP}), which is a convex function of $\bx$. 
\end{principle}

We can now express the dual formulation in the language of optimal transport theory,  \cite{ambbook, villanioldnew}. 

\begin{definition}\label{potden}
The {\em potential density} $\nu(t,\bx)\in{\cal P}([0,\tau)\times\Omega)$ associated to the system (\ref{sginc}) is the push forward of the Lebesgue measure of the domain $\Omega$ through the map $ \mathbf T $  given by (\ref{bbgeovariables}):
\begin{equation}
\nu = \mathbf T \push \chi _\Omega. \end{equation}
This means that  the measure $\nu$ is defined by
$$
\nu(B)=|\mathbf T^{-1}(B)|, \quad \forall\;B\subset \R^3 \; Borel\; set,
$$
and satisfies the change of variable formula
$$
 \int_{\Omega}   f(\mathbf T(\bx ))d\bx= \int_{\R^3}   f(\by)d\nu(\by)\qquad \forall f\in{\bf C}_c(\R^3).
$$
\end{definition}
We can now rephrase Cullen's stability principle as the requirement that $\mathbf T$ which minimises (\ref {bbenergyrewrite}) is the {\em optimal map} in the transport of $\chi _\Omega $ to $\nu $ with respect to the cost function $c(\mathbf x, \mathbf y)$ given by
\begin{equation}\label{bbcost}
c(\mathbf x, \mathbf y) = \frac{1}{2}\{ |x_1 - y_1|^2 + |x_2 - y_2|^2 \} - x_3y_3.
\end{equation}
Brenier's polar factorization theorem \cite{brenier} ensures the existence of a unique such optimal map, and guarantees that this optimal map, for each fixed time $t$, is of the form $\mathbf T = \nabla P$ with $P$ a convex function of the space variable $\bx$.  

Hence defining $\mathbf T$ as in (\ref{bbgeovariables}) and $P$ as in (\ref{bbP}), we can use 
 the fact that $D_t\mathbf x= \mathbf u$, to rewrite (\ref{sginc})-(\ref{bbbound}) as the following system of equations for $P(t,\bx)$,  $\mathbf u(t,\bx)$:
\begin{eqnarray}
&&\label{cf2.3i}D_t\mathbf T(t, \mathbf x) = J(\mathbf T(t, \mathbf x) - \mathbf x),\\
&& \label{cf2.3ii}\nabla \cdot \mathbf u = 0,\\
&& \label{cf2.3iii}\mathbf T(t, \mathbf x) = \nabla P(t, \mathbf x),\\
&& \label{cf2.3iv}\mathbf u \cdot \mathbf n = 0 \textrm{ on $[0, \tau ) \times \partial\Omega $},
\end{eqnarray}
with initial condition
\be
 \label{cf2.3v}P(0, \mathbf x) = P_0(\mathbf x):=  \frac{1}{2}(x_1^2 + x_2^2) + p_0(\mathbf x) \textrm{ in $\Omega $},
 \ee
where the symplectic matrix $J$ is defined by 
\begin{equation}
\label{bbJ}
J=\left(\begin{array}{ccc}
0&-1&0
\\
1&0&0
\\0&0&0
\end{array}\right)
\end{equation}
We now write (\ref{cf2.3i})-(\ref{cf2.3v}) in Lagrangian form.  We define the Lagrangian flow map $\mathbf F(t, \mathbf x)$ corresponding to the velocity $\mathbf u$, i.e. 
\[ \frac{\partial }{\partial t}\mathbf F(t, \mathbf x) = \mathbf u(t, \mathbf F(t, \mathbf x)), \qquad \mathbf F(0, \mathbf x) = 0,\]
and can then rewrite (\ref{cf2.3i}), (\ref{cf2.3iii}), as first done in \cite{feldman}, in the form
\begin{equation}\label{bbevolution}
\frac{\partial }{\partial t}\bZ (t, \bx ) = J(\bZ (t, \bx ) - \bF (t, \bx )),\qquad \bZ (t, \bx ) = \nabla P(t, \bF (t, \bx )).
\end{equation}

The incompressibility condition and the  boundary condition can then be reformulated  as
\begin{equation}\label{bblaginc1}
\mathbf F(t, \cdot ) \push \chi _{\Omega} = \chi _{\Omega}\Longleftrightarrow detD\bF (t, \bx ) = 1,
\end{equation}
where $D\mathbf F$ is the Jacobian matrix of $\mathbf F$. Hence $\mathbf F(t, \cdot )$ is a volume preserving mapping of $\Omega $. 
 
 \medskip
Using (\ref{bbevolution}), it is possible to derive an evolution equation for $\nu (t, \mathbf y)$ in dual space.  Namely, for any $\xi\in C_c^1([0,\tau)\times \R^3)$,
\begin{equation}\label{bb17}
\int_{[0, \tau ) \times \mathbb{R}^3} \! \left( \frac{\partial }{\partial t}\xi (t, \mathbf y) + \mathbf w(t, \mathbf y)\cdot \nabla \xi (t, \mathbf y)\right) \nu (t, \mathbf y) \, d\mathbf y dt + \int_{\mathbb{R}^3} \! \xi (0, \mathbf y)\nu (0, \mathbf y) \, d\mathbf y = 0,
\end{equation}
where  the dual velocity $\mathbf w$ is defined (and automatically divergence-free, by its definition)  by
\begin{equation}\label{bbw}
\mathbf w(t, \mathbf y) = J(\mathbf y - \nabla P^*(t, \mathbf y)) \Longrightarrow \nabla\cdot \mathbf w=0.
\end{equation}
with $P^*$ denoting the Legendre transform of $P$:
\begin{equation}P^* = \sup _{\mathbf x \in \Omega } \{ \mathbf x \cdot \mathbf y - P(t, \mathbf x) \}.
\label{legendredef}
\end{equation} 
Equation (\ref{bb17}) is the weak formulation of the transport equation 
\begin{equation}\label{bbdualcont}
\frac{\partial }{\partial t}\nu (t, \mathbf y) + \nabla \cdot (\mathbf w(t, \mathbf y)\nu (t, \mathbf y)) = 0.
\end{equation}

Combining (\ref{bbdualcont}), (\ref{bbw}) and the weak formulation of the Monge-Amp\`{e}re equation (\ref{bblaginc1}) yields the {\em semi-geostrophic equations in dual variables}
\begin{eqnarray}
&&\label{bbdual1} \frac{\partial }{\partial t}\nu (t, \mathbf y) + \nabla \cdot (\mathbf w(t, \mathbf y)\nu (t, \mathbf y)) = 0, \qquad \;\;(t,x)\in [0, \tau ) \times \mathbb{R}^3 ,\\
&&\label{bbdual2} \mathbf w(t, \mathbf y) = J(\mathbf y - \nabla P^*(t, \mathbf y)),\qquad \qquad\qquad (t,x)\in [0, \tau ) \times \mathbb{R}^3 , \\
&&\label{bbdual3} 
\nabla P(t, \cdot ) \push \chi _\Omega = \nu (t, \cdot ), \qquad \textrm{  }\qquad\qquad \qquad \;\;  t \in [0, \tau ),
\end{eqnarray}
where $J$ is defined by (\ref{bbJ}) and  $P^* $ by (\ref{legendredef}); $\nabla P(t, \cdot )$ is the unique optimal transport map of $\chi _\Omega$ to $\nu(t,\cdot)$.


Equation (\ref{bbdual3}) expresses the energy minimisation requirement, hence it is a precise mathematical formulation of Cullen's principle.    
Equations (\ref{bbdual1})-(\ref{bbdual3}) are supplemented with the  initial condition 
\begin{equation}\label{bbinitialnu}
\nu (0, \cdot ) = \nu _0(\cdot ), \qquad \mathbf y \in B(0,r)\subset\mathbb{R}^3.
\end{equation}
Note that we require that  $\nu _0$ is a given measure  with {\em compact support} contained in some ball $B\subset\mathbb{R}^3$.

\subsubsection*{The proof of Benamou and Brenier}
To prove the existence of weak solutions of the  system (\ref{bbdual1})-(\ref{bbinitialnu}), the following strategy was introduced in \cite{benamou}:

\begin{itemize}
\item
Given the  compactly supported, absolutely continuous measure $\nu(t,\by)$ at a given fixed time $t$,  compute the velocity field $\mathbf w$ from (\ref{bbdual3}) and (\ref{bbdual2}). 
\item
 In  order to advect $\nu $ in time using (\ref{bbdual1}), the system is discretised in time.  Then $\mathbf w$ is used to advect $\nu $ to the next time step, using the transport equation (\ref{bbdual1}).  Due to the way in which $\mathbf w$ is constructed, we have that $\mathbf w \in L^\infty _{loc}([0, \tau )\times \mathbb{R}^3)$ and $\mathbf w \in L^\infty ([0, \tau ); \, BV_{loc}(\mathbb{R}^3))$. 
The measure $\nu$ remains compactly supported within a ball whose radius depends on time.
 \item
To solve the transport equation, one must also  use a sequence of regularised problems, with Lipschitz continuous velocity field,   that approximates $\mathbf w$. For the approximating problems, the transport equation is uniquely solvable.   Then, using the stability property of polar factorisation, one can show that these approximate solutions converge to solutions of the system (\ref{bbdual1})-(\ref{bbinitialnu}).  
 \end{itemize}

This strategy gives a proof of the main result \cite[Theorem 5.1]{benamou}; our slightly more general statement  is taken from \cite[Theorem 2.3]{feldman}:
\begin{theorem}\label{bbdualth}
Let $\Omega \subset \mathbb{R}^3$ be an open bounded set such that $\overline{\Omega} \subset B(0, S)$, where $B(0, S)$ is an open ball of radius $S$ centred at the origin.  Let $P_0(\mathbf x)$ be a convex bounded function in $B(0, S)$ satisfying
\begin{equation}\label{cf2.13}
\nu _0 := \nabla P_0 \push \chi _\Omega \in L^q(\mathbb{R}^3)
\end{equation}
for some $q>1$.  Then, for $\tau > 0$, there exist functions $\nu $ on $[0, \tau ) \times \mathbb{R}^3$, $P$ on $[0, \tau ) \times \Omega $ such that $(\nu , P)$ satisfy (\ref{bbdual1})-(\ref{bbdual3})and the initial condition (\ref{bbinitialnu}) in the weak sense. In addition,
\renumerate{ \item $\nu $, $P$ satisfy
\begin{eqnarray*}
&& \nu \in L^\infty ([0, \tau ); L^q(\mathbb{R}^3))\cap C([0, \tau ); L_w^q(\mathbb{R}^3)),\\
&& P \in L^\infty ([0, \tau ); W^{1, \infty }(\Omega )) \cap C([0, \tau ); W^{1, r}(\Omega )),\quad P(t, \cdot )\; is\; convex\; in \;\Omega; 
\end{eqnarray*}
where $r \in [1, \infty )$ and $C([0, \tau ); L_w^q(\mathbb{R}^3))$ is the set of all measurable functions $\mu (t, \mathbf y)$ on $[0, \tau ) \times \mathbb{R}^3$ such that $\mu _{(t)}(\cdot ) = \mu (t, \cdot ) \in L^q(\mathbb{R}^3)$ for any $t \in [0, \tau )$ and, for any $\{ t_k\} _{k = 1}^\infty , \, t_* \in [0, \tau )$ satisfying $\lim _{k \to \infty } t_k = t_*$, we have $\mu _{(t_k)} \rightharpoonup \mu _{(t_*)}$ weakly in $L^q(\mathbb{R}^3)$ (narrowly if $q = \infty $);

\item $\textrm{ for all } t \in [0, \tau ),\;\; supp(\nu (t, \cdot )) \subset B(0, R_0), $
where $R_0 = S(1 + \tau )$;
%

\item $P^* = \sup _{\mathbf x \in \Omega } \{ \mathbf x \cdot \mathbf y - P(t, \mathbf x) \} $ satisfies
\begin{eqnarray*}
&& P^*(t, \cdot ) \textrm{ is convex in $\mathbb{R}^3$ for any $t \in [0, \tau )$},\\
&& P^* \in L^\infty _{loc}([0, \tau ) \times \mathbb{R}^3),\\
&& \nabla P^* \in L^\infty ([0, \tau ) \times \mathbb{R}^3; \mathbb{R}^3) \cap C([0, \tau ); L^r(B(0, R); \mathbb{R}^3)),
\end{eqnarray*}
for any $R>0$ and any $r \in [1, \infty )$.  Moreover,
\[ \| \nabla P^*(t, \cdot )\| _{L^\infty (\mathbb{R}^3)} \leqslant S \qquad \textrm{ for every } t \in [0, \tau ).\]

\item $\mathbf w \in L^\infty _{loc}([0, \tau )\times \mathbb{R}^3)$, $\mathbf w \in L^\infty ([0, \tau ); \, BV_{loc}(\mathbb{R}^3))$.



 }
 
\end{theorem}


 \begin{remark}
 The original result of \cite{benamou} makes the assumption $q>3$ in Theorem \ref{bbdualth}.  Lopes Filho and Nussenzveig Lopes \cite{lopes} extended this result to $q>1$.  Loeper \cite{loeper} extended this result further, proving existence and stability of measure valued solutions. In \cite{faria}, Faria \emph{et al.} have extended the results of \cite{feldman} for the incompressible equations to the case of an  initial potential density $\nu _0$ in $L^1$. Faria has recently done the same for the compressible system as well, \cite{faria2} .
 
 In view of these results, we will include the case $q=1$ in our main statements below.
 \end{remark}

The strategy employed to prove Theorem \ref{bbdualth} can be adapted to prove existence of weak solutions in dual space for the compressible equations \cite{CGP1,dkgthesis}. In this paper, we will prove an analogous result for the case of  a free boundary condition, using a modification of the original strategy that does not explicitly require the time discretization argument of \cite{benamou}, but relies instead on the theory of Hamiltonian ODEs of \cite{ambgangbo}, summarised in the Appendix.  This basic structure of proof was already used in  \cite{CGP1}.

\section{The incompressible free boundary problem}
\setcounter{equation}{0}

In this section, we study the problem obtained when  the rigid boundary condition  (\ref{bbbound}) considered in \cite{benamou} is replaced by a more physically relevant free boundary condition.   To model this situation, the equations (\ref{sginc}) are to be solved in $[0, \tau ) \times \Omega_h (t)$, where the domain $\Omega_h (t) \subset \mathbb{R}^3$ is time-dependent and represents the region occupied by the fluid at time $t$:  
\begin{equation}\label{cgfreeOmega}
\Omega_h (t) = \{ (x_1, x_2, x_3) \in \mathbb{R}^3 : (x_1, x_2) \in \Omega _2, 0 \leqslant x_3 \leqslant h(t, x_1, x_2) \} .
\end{equation}
Here $\Omega _2 \subset \mathbb{R}^2$ is a fixed bounded domain with rigid wall boundary conditions, while  $h(t,x_1,x_2)$ is unknown and represents the free boundary.  

The incompressibility of the flow can be formulated as the requirement that $|\Omega_h (t)|$ remains constant for all $t \in [0, \tau )$, where $|\cdot |$ denotes the three-dimensional Lebesgue measure.  In what follows, we normalise the measure so that 
$$|\Omega_h(t)| = 1\quad  for \; all \;\;t<\tau.$$

We denote by $\sigma_{h}(t,\bx)\in\mathcal{P}_{ac}(\mathbb{R}^3)$ the probability measure defined on $\R^3$  by 
\begin{equation}\label{sigmah}
 \sigma_{h}(t,\bx)= \chi _{\Omega _{h(t) }}(\bx), \qquad \int_{\R^3}\sigma_h(t,\bx)d\bx=1\quad \forall t<\tau.
\end{equation}

We make no a-priori assumption that $h(t,x_1,x_2)$ is a well defined, single valued function, since in principle the free boundary could develop an overhanging profile. Hence our notation in (\ref{cgfreeOmega}) is not   well defined. However,  we will show that the solution indeed corresponds  to a well-defined function, so the  abuse of notation in our definition of the domain is ultimately justified.
  
The flat rigid bottom of the domain is defined by $x_3 = 0$.



The  boundary conditions we consider are 
\begin{eqnarray}\label{cgspeedboundary}
&&\mathbf u \cdot \mathbf n = 0 \qquad \bx\in \partial \Omega_h (t)\setminus \{x_3=h\},
\\
\label{cgboundcon}
&&\left\{
\begin{array}{l}\partial_th+u_1\frac{\partial h}{\partial x_1}+ u_2\frac{\partial h}{\partial x_2}=u_3, \\ p(t, x_1, x_2, h(x_1, x_2)) = p_h,\end{array}
\right.\quad \bx\in \partial \Omega_h (t):\; x_3=h(t,x_1,x_2),
\end{eqnarray}
where $p_h$ is  a prescribed constant; for convenience henceforth we take $p_h=0$.

\medskip
In what follows, we first state the results of \cite{gangbo}, obtained by taking advantage of  the additional assumption of constant density. This assumption reduces the dimensionality of the problem, so that the governing equations are transformed to the shallow water  system.

We then consider variable density and the incompressible three-dimensional problem, and prove our main result.

\subsection{Constant density - the 2-D shallow water equations}\label{cg}
When the density is assumed constant, the system (\ref{sginc}) describing the flow of an incompressible fluid reduces to the \emph{two-dimensional semi-geostrophic shallow water equations}:
\begin{eqnarray}
&&\label{shallmom1} D^{(2)}_t u_1^g -  u_2 + \frac{\partial h}{\partial x_1} = 0,\\ 
&&\label{shallmom2} D^{(2)}_tu_2^g +  u_1 + \frac{\partial h}{\partial x_2} = 0,\\ 
&&\label{shallcont} \frac{\partial h}{\partial t} + \nabla _2 \cdot (h\mathbf u_2) = 0,\\ 
&&\label{shallgeo} u_1^g = -\frac{\partial h}{\partial x_2}, \qquad u_2^g = \frac{\partial h}{\partial x_1}, 
\end{eqnarray}
where $\mathbf u_2 = (u_1, u_2)$, $D_t^{(2)}=\partial_t+\mathbf u_2\cdot\nabla$,  and all equations are to be solved for $(t, \mathbf x) \in [0, \tau ) \times \Omega _2$.
  The system (\ref{shallmom1})-(\ref{shallgeo}) is to be considered with the prescribed initial and boundary conditions 
\begin{equation}\label{shallcon} 
\mathbf u_2 \cdot \mathbf n = 0 \quad \textrm{ on } [0, \tau ) \times \partial \Omega _2, \qquad h(0, \cdot ) = h_0(\cdot ) \quad \textrm{ in } \Omega _2.
\end{equation}
Note that the evolution of the free boundary $h(t,\bx)$ is now explicitly part of the system of governing equations, which are posed in the {\em fixed} domain $\Omega_2$.

The 2-D geostrophic energy associated with the flow is defined by
\begin{equation}\label{shallenergy}
E_2 = \int_{\Omega _2 }\! \bigg( \frac{1}{2}( (u_1^g)^2 + (u_2^g)^2)h + \frac{1}{2}h^2 \bigg)  \, dx_1dx_2.
\end{equation}

The dual system in Lagrangian  coordinates, obtained after passing to the dual coordinates $y_1=x_1+u_g^2$, $y_2=x_2-u_g^1$, is given by
  \begin{eqnarray}
&&\label{shalldual1} \frac{\partial }{\partial t}\nu (t, \mathbf y) + \nabla _2 \cdot (\mathbf w(t, \mathbf y) \nu (t, \mathbf y)) = 0, \qquad J_2=\left(\begin{array}{cc}
0&-1\\1&0\end{array}\right),\\
&&\label{shalldual2} \mathbf w (t, \mathbf y) =  J_2 (\mathbf y - \nabla _2 P^*(t, \mathbf y)), \qquad \textrm{ in } [0, \tau ) \times \mathbb{R}^2, \\
&&\label{shalldual3} \nabla _2 P(t, \cdot ) \textrm{\#}h(t, \cdot ) = \nu (t, \cdot ) \qquad \textrm{ for any } t \in [0, \tau ),\\
&&\label{shalldual4} P(t, \mathbf x) =   h(t, \mathbf x) + \frac{1}{2}(x_1^2 + x_2^2) , \qquad \textrm{ in } [0, \tau ) \times \Omega _2 ,\\
&&\label{shalldual5} \nu (0, \mathbf y) = \nu _0(\mathbf y) \qquad \textrm{ given, compactly\;supported}.
\end{eqnarray}

The main theorem of \cite{gangbo} is summarised below.
\begin{theorem}\label{cgdualth}
Let $\Omega _2 \subset \mathbb{R}^2$ be an open connected set. Let $r$ be given, $1\leq r<\infty$. Assume that $\nu_0\in L^r(\R^2)$, $h_0\in L^1(\R^2)$ are two probability density functions,    such that  {\em support}$(\nu_0) \subset B(0, S)$, where $B(0, S)$ is an open ball of radius $S$ centered at the origin.  Assume also that the function $P_0(\bx)= |x|^2/2+h_0(\bx)$ can be extended to  a convex bounded function in $\R^2$ and that $\nu_0$, $h_0$ satisfy
\begin{equation}\label{cg2.13}
\nu _0 = \nabla  P_0 \push h _0.
\end{equation}
Then, for $\tau > 0$, there exist functions $\nu $ on $[0, \tau ) \times \mathbb{R}^2$, $P$ on $[0, \tau ) \times \Omega _2 $ such that
$(\nu,P)$ satisfy (\ref{shalldual1})-(\ref{shalldual5}) and the initial condition (\ref{shalldual5}) in the weak sense.
In addition $\nu $, $P$ satisfy the regularity stated in  (i)-(iv)  of Theorem \ref{bbdualth}.

\end{theorem}

\subsection{Variable density - the incompressible free boundary problem in 3-D }

We now consider the incompressible semi-geostrophic system (\ref{sginc})  in the region $\Omega_h (t) $ given by (\ref{cgfreeOmega}), with  boundary conditions  (\ref{cgspeedboundary})-(\ref{cgboundcon}). 

%

\medskip
The energy associated with the flow is the geostrophic energy  defined by
\begin{equation}\label{myincenergy}
E = \int_{\Omega_h}\! \bigg( \frac{1}{2}((u_1^g)^2 + (u_2^g)^2) + \rho x_3 \bigg) \, d\mathbf x.
\end{equation}
By a formal but straightforward calculation, it can be shown that, as expected, this   energy integral is conserved in time.
\begin{proposition}\label{energyconserve}
The system (\ref{sginc})-(\ref{hinitial}) conserves the energy integral in (\ref{myincenergy}).
\end{proposition}

Similarly,  a  formal argument shows that geostrophic and hydrostatic balance can be characterised as a stationary point of the energy in (\ref{myincenergy}) with respect to a particular class of variations, supporting the validity of Cullen's stability principle also in this case. 

\begin{remark}[Support of the density $\rho(t,\bx)$]
We can assume that there exists $\delta>0$ such that the density $\rho(t,\bx)$ satisfies
\be
\delta<\rho(t,\bx)<\frac 1 \delta,\qquad \bx\in\Omega_h,\; \mbox{uniformly for }\; t<\tau.
\label{densityest}\ee
This follows from assuming the bound at time $t=0$ and employing the third of equations (\ref{sginc}).
 The full arguments are presented  in \cite{dkgthesis}.
\end{remark}

Note that the incompressibility condition as expressed by (\ref{sigmah}) and the conservation of energy  (\ref{myincenergy}) imply that any sufficiently regular $h(t,\cdot)$ which is a solution of the system  has to satisfy
\be
 h(t,\cdot)\in L^1\cap L^2(\Omega_2),
 \label{hass}
 \ee
 at least if it is assumed that $\rho(t=0)$ satisfies the bound (\ref{densityest}), and that the energy $E$ is initially bounded. 
  
 Indeed, 
\be
\|h\|_1=\int_{\Omega_2}h(x_1,x_2)dx_1dx_2=\int_{\Omega_2}\int_0^h d\bx=\int_{\Omega}d\sigma_h=1,
\label{l1}\ee
and
$$
\|h\|_2^2=\int_{\Omega_2}h^2(x_1,x_2)dx_1dx_2=\int_{\Omega_2}\left[\int_0^h2x_3dx_3 \right]dx_1dx_2\leq \frac 2 {\delta_\rho}\int_{\Omega_2}\left[\int_0^h
\rho x_3dx_3 \right]dx_1dx_2
$$
\be
\leq\frac 2 {\delta_\rho} \int_{\Omega_h}\! \bigg( \frac{1}{2}((u_1^g)^2 + (u_2^g)^2) + \rho x_3 \bigg) \, d\mathbf x=\frac 2 {\delta}E:=C_0.
\label{l2}\ee

We also assume that there exists a constant  $H>0$ such  that for every admissible $h(t,\cdot)$,
 \be
 \Omega_h\subset  \Omega_2\times [0,H):=\Omega_H.
 \label{omegaH}\ee 
  This assumption will be justified by our solution procedure.   


\subsubsection{Dual formulation}\label{secnewgeo}
In what follows, we assume that $\Lambda \subset \mathbb{R}^3$ is an open bounded set.  Indeed, we assume there exists $R_0>0$ such that 
\be
\Lambda\subset  \Lambda_2\times [-R_0,0), \quad R_0>0, \;\Lambda_2\subset \R^2\;bounded.
\label{Lambda0}\ee
This bound follows from the bound (\ref{densityest}) on $\rho(\bx,t)$, and from the fact that $\Lambda_2$ can be assumed to remain bounded. The latter is guaranteed by condition (H1), see section \ref{newmain}.

\medskip
The change of variables to the geostrophic coordinates, for each fixed $h(x_1,x_2)$ describing the domain,  is defined in this case by
\[ \mathbf T:\Omega_h(t)\to\Lambda,\quad \mathbf T(t, \mathbf x) = (T_1(t, \mathbf x), T_2(t, \mathbf x), T_3(t, \mathbf x)) = (y_1, y_2, y_3), \]
where
\begin{equation}\label{newchange}
T_1(\bx) = x_1 + u_2^g, \qquad T_2(\bx) = x_2 - u_1^g, \qquad T_3(\bx) = -\rho.
\end{equation}
This definition of the mapping $\mathbf T$, and the bound (\ref{Lambda0}), imply that, for all $t<\tau$, the geostrophic velocity $(u_1^g, u_2^g)$ remains bounded. 

We will denote the inverse of $\mathbf T$ by $\mathbf S$ (see Theorem \ref{new3.3} below);
\[ \mathbf S(t, \mathbf y ) = (S_1(t, \mathbf y ), S_2(t, \mathbf y ), S_3(t, \mathbf y )) = \mathbf T^{-1}(t, \mathbf y),\quad \by\in\Lambda.\]

\smallskip
We show next that, as in the rigid boundary case, the problem can be formulated as an optimal transport problem, whose solution is given by the gradient of a convex function.

We use (\ref{newchange}) to rewrite the energy in (\ref{myincenergy}), at fixed time $t$, as the following functional in dual space:
\begin{eqnarray}
 E[\mathbf T, h] &&= \int_{\Omega _h}\! \left[ \frac{1}{2}\{|x_1 - T_1(\bx)|^2 + |x_2 - T_2(\bx)|^2 \} - x_3T_3(\bx) \right] \, d\mathbf x \label{en1}
\end{eqnarray}
The following definition is the analogue of  Definition \ref{potden}. 

\begin{definition}
Given $\sigma_h$ as in (\ref{sigmah}), define the potential density $\nu := \mathbf T \textrm{\#}\sigma_h \in \mathcal{P}_{ac}(\Lambda)$ associated with the flow described by  
 (\ref{sginc})-(\ref{cgboundcon}) as the push forward of the measure $\sigma_h\in  \mathcal{P}_{ac}(\mathbb{R}^3)$ under the map $\mathbf T$ given by (\ref{newchange}). \end{definition}
 \begin{remark}[Support of the potential density $\nu(\bx,t)$]
We show below that the potential density  $\nu(t,\mathbf y)$ must satisfy the evolution (\ref{newdualinc1}), Assuming that  at  time $t=0$ the initial potential density  $\nu_0$ has compact support in $\mathbb R^3$, we can deduce that   $supp(\nu)$ is contained in a bounded open set $\Lambda $, depending on the time interval length $\tau $, such that $\overline{\Lambda} \subset \mathbb{R}^2 \times [-\frac{1}{\delta},-\delta]$, for some $\delta$ with $0<\delta<1$. This follows from a standard fixed-point argument; see, for example, \cite{cullen,loeper}.
\end{remark}
 
Define the functional
\begin{equation}\label{newenergyminfirst}
E_\nu (\sigma_h) = \inf _{{\mathbf T}:\,{\mathbf T} \textrm{\#} \optfreem_h = \nu }\int_{\mathbb{R}^3} \! c(\mathbf x, {\mathbf T}(\mathbf x))\sigma_h(\mathbf x)\, d\mathbf x,
\end{equation}
where $\optfreem_h $ is defined in (\ref{sigmah}) and the cost function  $c$ is given by
\begin{equation}\label{newcost}
c(\mathbf x, \mathbf y) = \left [\frac{1}{2}\{|x_1 - y_1|^2 + |x_2 - y_2|^2 \} - x_3y_3 \right ] .
\end{equation}

\begin{principle}[Cullen's stability principle]
At each fixed time $t$, the pair $(\sigma_{\bar h} , \overline{\mathbf T})$ corresponding to a solution of   (\ref{sginc}) with boundary conditions (\ref{cgboundcon}) minimises the energy (\ref{en1}) amongst all pairs $(\sigma_h , {\mathbf T})$ where  $\sigma_h$ is given by (\ref{sigmah})  and ${\mathbf T} \textrm{\#} \sigma_h = \nu $.  

Namely,  given $\nu \in \mathcal{P}^2_{ac}(\Lambda )$,  a stable solution corresponds to the following minimal value for the  energy:
\begin{equation}\label{newstablefirst} {\cal E}(t,\nu) =  \inf _{\sigma_h\in{\mathcal{H}} }E_\nu (\sigma_h ) = \inf _{\sigma_h\in{\mathcal{H}} } \left\{ \inf _{{\mathbf T}:\,{\mathbf T} \textrm{\#} \optfreem_h = \nu }\int_{\mathbb{R}^3} \! c(\mathbf x, {\mathbf T}(\mathbf x)) \optfreem_h(\mathbf x) \, d\mathbf x \right\} ,
\end{equation}
where ${\cal H}\subset \mathcal{P}_{ac}(\R^3)$ is an appropriate subset of $\mathcal{P}_{ac}(\R^3)$.
\end{principle}

%

%
\subsubsection{Lagrangian formulation and statement of the main theorem}

We  formulate the semi-geostrophic system in dual variables in Lagrangian form, in a way entirely analogous to the rigid boundary case. This yields
\begin{eqnarray}
&&\label{newdualinc1} \frac{\partial \nu }{\partial t} + \nabla \cdot (\nu \mathbf w) = 0, \quad \textrm{ in } [0, \tau ) \times \Lambda , \\
&&\label{newdualinc2} \mathbf w(t, \mathbf y) = J(\mathbf y - \nabla P^*(t, \mathbf y)),  \quad \textrm{ in } [0, \tau ) \times \Lambda , \\
\label{newdual trans}&&\nabla P\push\sigma_{ h}=\nu,\quad\nabla P(t,\cdot)\,\textrm{ is the unique optimal transport map, and}
\\
\label{newdualinc5}&& \sigma_{h}  \textrm{ minimises } E_{\nu(t, \cdot )} (\cdot ) \textrm{ over } \Freem ,  \qquad t \in [0, \tau ).
\end{eqnarray}
Here, $P^*$ denotes the Legendre transform of the (convex) function $P$ and ${\cal H}$ denotes an appropriate minimisation space, which we define in the next section, see (\ref{calH}).
   
At each fixed time $t<\tau$, the unknowns in this system are the fluid profile $h(t,x_1,x_2)$  and the geopotential  $P(t,\bx)$. We can assume that $h(t,x_1,x_2)$ is a well defined function of $(x_1,x_2)\in\Omega_2$, an assumption justified by the result of Lemma \ref{lemmasingleval} below.

Given $P(t,\bx)$ and $h(x_1,x_2,t)$, it is possible to reconstruct $\nu=\nabla P\push\sigma_h$. Moreover, we show in Proposition \ref{Pp} below  that the pressure $p(t,\bx)$ is obtained from the solution  $P(t,\bx)$ of the system  through the relation
\be
\label{newdualinc4.5}
 p(t,\mathbf x) = P(t,\mathbf x) - \frac{1}{2}(x_1^2 + x_2^2),  \quad (t,\bx) \in [0, \tau ) \times \Omega.
\ee
The system is to be  solved, in the weak sense of (\ref{bb17}), given the following initial conditions 
\begin{equation}\label{hinitial}
h(0,\cdot) = h_0(\cdot ) \in W^{1,\infty}(\Omega _2 ), \qquad  (x_1, x_2) \in \Omega _2,
\end{equation}
\be
\label{newdualinc6}
\nu (0, \cdot ) = \nu _0 (\cdot )\;\; \text{\it compactly supported probability density in } L^r, \, r \in [1, \infty ),\ee
\be
P(0,\bx)=P_0(\bx)\in W^{1,\infty}(\Omega_{h_0}),\label{P0in}
\ee
satisfying the compatibility condition
\be
\nabla P_0\push\sigma_{h_0}=\nu_0.
\label{compds}
\ee


\smallskip
It is not difficult to show that,  formally, (\ref{newdualinc1})-(\ref{newdualinc6}) yields a stable solution of (\ref{sginc}), see \cite{dkgthesis}:
\begin{lemma}\label{formalyield}
A sufficiently regular solution of (\ref{newdualinc1})-(\ref{newdualinc6}) yields a solution of (\ref{sginc}) with initial condition (\ref{hinitial}) and boundary conditions (\ref{cgspeedboundary})-(\ref{cgboundcon}).
\end{lemma}

We can now state the main theorem. The proof is presented in section \ref{newmain}.

\begin{theorem}\label{new5.5}

Let $1 \leq r < \infty $ and let $\nu _0 \in L^r (\Lambda _0 )$ be an initial potential density with support in $\Lambda _0$, where $\Lambda _0 \subset \Lambda_2\times[-R_0,0)$ with $R_0>0$ and $\Lambda_2$ is a bounded open set in $\mathbb{R} ^2$. 
Let $c(\cdot ,\cdot )$ be given by (\ref{newcost}). 

 Then the system of semi-geostrophic equations in dual variables (\ref{newdualinc1})-(\ref{newdualinc6}) with given conditions (\ref{hinitial}), (\ref{P0in}) satisfying the compatibility condition (\ref{compds}), has a stable weak solution $(h,  P)$ such that $\nu  = \mathbf T \textrm{\#} \sigma_h $, where $\mathbf T = \nabla P$, $\sigma_h = \chi _{\Omega _2 \times [0, h ]}$, and $\nu$ has compact support.
 
 This solution satisfies:
 \begin{enumerate}[label=(\textit{\roman{*}}), ref=\textit{(\roman{*})}]
	\item \[ \nu (\cdot , \cdot ) \in L^\infty ([0, \tau ); L^r( \Lambda) ), \qquad \left\| \nu (t, \cdot )\right\| _{L^r (\Lambda )} \leqslant \left\| \nu _0 (\cdot )\right\| _{L^r (\Lambda )}, \qquad \forall\; t \in [0, \tau ],
\]
\item  \[ P (t, \cdot ) \in L^\infty ([0, \tau ); W^{1, \infty } (\Omega _2)), \qquad \left\| P(t, \cdot )\right\| _{W^{1, \infty } (\Omega _{\freemin})} \leqslant C = C(\freemin, \Lambda, c(\cdot , \cdot )),\] \[ \forall\;t \in [0, \tau ], \]
\item \[ h(t, \cdot ) \in W^{1, \infty }(\Omega _2 ), \qquad \textrm{ for all } t \in [0, \tau ), \]
\end{enumerate}
where  $\Lambda $ is a bounded open domain in $\mathbb{R}^3$ containing $supp (\nu )$ for all $ t\in[0,\tau)$.
\end{theorem}

\subsection{The minimisation problem (\ref{newstablefirst})}
\label{Hsect}

In the rest of this section, we fix the time $t\in(0,\tau)$ and often drop the explicit dependence on it from the equations. 

\smallskip
Our aim is  to  prove existence and uniqueness of a minimiser of the functional $E_\nu(h)$ given by (\ref{newstablefirst}). We do not follow the strategy employed for the proof of the analogous result for the 2-dimensional problem. Indeed, in our case it does not seem straightforward to prove that the energy functional is strictly convex with respect to $h$. To prove uniqueness of the minimiser,  we will consider the Monge-Kantorovich formulation of the problem, following what done in \cite{sedjro} for the more difficult case of a forced axisymmetric flow.

To be able to prove that the minimisation problem (\ref{newstablefirst}) admits a solution, we first consider what conditions the problem imposes on the minimisation space ${\cal H}$. 

We start by  showing that, for every fixed value of $t<\tau$,  the minimiser has to correspond to a well defined, single-valued function $h(x_1,x_2)\in L^1\cap L^2(\Omega_2)$. 
\begin{lemma}\label{lemmasingleval}
 The minimiser  of (\ref{newstablefirst}) is given by a  $\sigma_h$ corresponding to  $\Omega_h=\Omega_2\times [0,h(x_1,x_2)]$ with $h(x_1,x_2)\in L^1\cap L^2(\Omega_2)$.
  \end{lemma}
\begin{proof}
Suppose that $\tildh $ is multi-valued and define the corresponding domain as  $\widetilde{\Omega}(t)$.  
  Define $\optfreem_{\tildh} := \chi_{\widetilde{\Omega}(t)} $. Choose a single valued function $h(x_1,x_2)$ such that $|\widetilde{\Omega}|=|\Omega_h|$, and transport map $\mathbf R$ such that
  \[ \mathbf R \push \optfreem_h= \optfreem_{\tildh} ,\]
where $\optfreem_h := \chi _{\Omega _h}$.  The existence of such a map $R$  is guaranteed by standard optimal transport results. 
 We choose $h$ in such a way that $\mathbf R$ can be expressed as $\mathbf R = (R_1(\mathbf x), R_2(\mathbf x), R_3(\mathbf x)) = (x_1, x_2, x_3 + \varphi (x_1, x_2, x_3))$ where $\varphi (x_1, x_2, x_3) \geqslant 0$ for all $(x_1, x_2, x_3) \in \Omega _2 \times [0, h]$.  Let $\nu \in \mathcal{P}_{ac}(\Lambda )$ and let $\widetilde{\mathbf T}$ denote the optimal map in the transport of $\optfreem_{\tildh}$ to $\nu $ with cost function (\ref{newcost}).  Then, since $\widetilde{\mathbf T} \circ \mathbf R \push \optfreem_h = \nu $ and $\widetilde{T}_3$ is negative, 
 we have
\begin{eqnarray*}
E_\nu (\tilde{\sigma}_h) &&= \inf _{{\mathbf T} \textrm{\#} \optfreem_{\tildh}  = \nu }\int_{\mathbb{R}^3} \! c(\mathbf x, {\mathbf T}(\mathbf x))\optfreem_{\tildh} (\mathbf x)\, d\mathbf x\\ &&= \int _{\mathbb{R}^3}\! c(\mathbf x, \widetilde{\mathbf T}(\mathbf x))\optfreem_{\tildh} (\mathbf x)\, d\mathbf x=\int _{\mathbb{R}^3}\! c(\mathbf R(\mathbf x), \widetilde{\mathbf T}\circ \mathbf R(\mathbf x))\optfreem_h (\mathbf x)\, d\mathbf x\\
&&\geqslant \int _{\mathbb{R}^3}\! c(\mathbf x, \widetilde{\mathbf T}\circ \mathbf R(\mathbf x))\optfreem_h (\mathbf x)\, d\mathbf x\geqslant \inf _{{\mathbf T} \textrm{\#} \optfreem_h = \nu }\int_{\mathbb{R}^3} \! c(\mathbf x, {\mathbf T}(\mathbf x))\optfreem_h (\mathbf x)\, d\mathbf x = E_\nu (\sigma_h ).
\end{eqnarray*}
Since $\tildh $ is an arbitrary multi-valued function, we conclude that any multi-valued upper boundary will have a corresponding single valued upper boundary which reduces the energy associated with the flow.  

The property that $h(x_1,x_2)\in L^1\cap L^2(\Omega_2)$ follows from (\ref{l1}) and (\ref{l2}).
\end{proof}

We now define the subset ${\cal H}$ of ${\cal P}_{ac}(\R^3)$ on which we minimise the energy.

\begin{definition}\label{calhdef}
%
%
%
We define the class
$\mathcal{H}\subset \mathcal{P}_{ac}(\mathbb{R}^3)$ by
\begin{equation}\label{calH}
	{\mathcal{H}} := \left \{
	 \sigma_{h}(t,\cdot)\in\mathcal{P}_{ac}(\mathbb{R}^3)\,\bigg|\;h\in {\cal H}_0\right\} , 
	\end{equation}
where the ${\cal H}_0\subset L^1\cap L^2(\Omega_2)$ is given as\be	
{\cal H}_0=\left\{h:[0, \tau ) \times \Omega _2 \to [0, \infty),\;h(t,\cdot)\in L^{1}(\Omega_2),\; \|h(t,\cdot)\|_1=1, \int_{\R^3}x_3d\sigma_h\leq 2C_0\right\}
\label{H0}\ee
where $\sigma_h$ is defined in (\ref{sigmah}) and $C_0$ is as in (\ref{l2}).

	\end{definition}

Our first aim is to show that the functional (\ref{newstablefirst}) admits  a minimizer in this space. 

\begin{proposition}
The functional (\ref{newstablefirst}) admits a minimising pair $(\sigma_h,T)$, where  $\sigma_h \in \mathcal{H}$, and T is the optimal map $T$ such that $\nu  = \mathbf T \textrm{\#} \sigma_h $.
\end{proposition}
	
	\begin{proof}
	Let
	\begin{equation}\label{newstablefirstcop} \tilde{\cal E}(\nu)  = \inf _{\sigma_h\in{\mathcal{H}} } \left\{ \inf _{r \in \Gamma(\sigma_h,\nu) }\int_{\mathbb{R}^3} \! c(\mathbf x, \mathbf y) \ r(\mathbf x,\mathbf y) d\mathbf x\ d\mathbf y \right\} ,
\end{equation}
where $\Gamma(\sigma_h,\nu)$ is the set of all bounded measures  $\mu \in \mathcal{P}\left(\Lambda \times \Omega_2\times [0,\infty) \right)$ with $ \pi_1\textrm{\#} \mu = \nu$, $ \pi_2\textrm{\#}\mu= \sigma_h$ where 
$\pi_1$ is the projection to $\Lambda$ and $\pi_2$ to $\Omega_2 \times [0,\infty) $.

By standard techniques using the narrow topology there exists an optimal $\mu\in \mathcal{P}\left(\Lambda \times \Omega_2\times [0,\infty) \right)$ approximated in the narrow topology by a sequence $\mu_n \in \Gamma(\sigma_{h_n},\nu)$, with $h_n \in \mathcal{H}_0$.

We consider ${\mathcal{H}} $ endowed with the weak-$L^1$-topology.
Note that the weak-$L^1$-convergence in ${\cal H}$ implies narrow convergence and that ${\cal H}$ is not closed in the weak-$L^1$-convergence. However, by Pettis criterium ${\cal H}$ is relative compact 
and hence w.l.o.g. we can assume that $\sigma_{h_n}$ converges $L^1$-weakly to $\sigma \in L^1 (\Omega_2 \times [0,\infty))$ with $0 \leq \sigma \leq 1$ and thus $\pi_2\textrm{\#}\mu = \sigma d\mathbf x$. Therefore  $\mu$ as an optimal transport plan from $\nu$ to $\sigma \in \mathcal{P}_{ac}(\mathbb{R}^3)$  is in fact given by an optimal map $T$.

Define $h := \int_0^\infty \sigma(x_1,x_2,x_3) dx_3$. It can easily be shown that $h \in \mathcal{H}_0$. 
We claim that $\sigma$ is actually equal to $\sigma_h$, so that there is a minimiser of the form required. 
The proof is similar to the proof of lemma \ref{lemmasingleval}, and is obtained by contradiction by showing that 
%
%
 $$
E_\nu(\sigma_h)= \inf _{{\mathbf T}:\,{\mathbf T} \textrm{\#} \sigma_h = \nu }\int_{\mathbb{R}^3} \! c(\mathbf x, {\mathbf T}(\mathbf x)) \optfreem_h(\mathbf x) \, d\mathbf x\leq 
 \inf _{{\mathbf T}:\,{\mathbf T} \textrm{\#} \sigma = \nu }\int_{\mathbb{R}^3} \! c(\mathbf x, {\mathbf T}(\mathbf x)) \sigma(\mathbf x) \, d\mathbf x=E_\nu(\sigma).
 $$
Indeed, consider the  transport map $\mathbf R$ such that
 $ \mathbf R \push \sigma_h = \sigma$,
where $\optfreem_h := \chi _{\Omega _h}$.  The existence of such a map $R$   is guaranteed by standard optimal transport results. 
It follows from the definition  of $h$ and the properties of $\sigma$ that $\mathbf R$  satisfies 
$
R_1(\bx)=x_1$, $R_2(\bx)=x_2$ and  $R_3(\bx)\geq x_3$.

Then as in the proof of Lemma \ref{lemmasingleval}, we deduce that $E_\nu(\sigma_h)\leq E_\nu(\sigma)$, hence the claim.

\end{proof}

\bigskip

In the remainder of this section, we will  show that, at each fixed time $t$,  there exists in fact a unique minimising pair $(\sigma_h,T)$ of the energy functional (\ref{newstablefirst}), with $\sigma_h\in{\cal H}$, and $T=\nabla P$ for a convex function $P$.

\subsubsection{Kantorovich formulation}
We assume that  $\nu \in \mathcal{P}_{ac}(\Lambda)$ is a given, compactly supported  density, and we consider the cost function $c(\bx,\by)$ defined by (\ref{newcost}).

The {\em Kantorovich dual} of the minimisation problem  (\ref{newstablefirst}) is the problem of maximising the functional  
\begin{eqnarray}\label{new26}
J_{(\optfreem , \nu )}(f, g) &=& \int_{\mathbb{R}^3} \! f(\mathbf x) \optfreem_h (\mathbf x)\, d \mathbf x + \int_{\Lambda } \! g(\mathbf y)\nu (\mathbf y) \, d \mathbf y,
\\
 f \in W^{1, \infty }(\mathbb{R}^3),\,g \in W^{1, \infty }(\Lambda ): && f(\mathbf x) + g(\mathbf y) \leqslant c(\mathbf x, \mathbf y) \textrm{ for all } (\mathbf x, \mathbf y) \in \mathbb{R}^3 \times \Lambda .
\nonumber \end{eqnarray}

It can be shown that the solution is unique and the key is in the notion of $c$-transfroms, defined by
$$
f^c(\by)=\inf_{\bx}(c(\bx,\by)-f(\bx)),\quad 
g^c(\bx)=\inf_{\by}(c(\bx,\by)-g(\by)).
$$
Then (see \cite{maroofi})  there exists a unique point $\bar\bx\in\R^3$ (respectively $\bar \by\in\R^3$), at which the infimum is attained, and which satisfies
\be
\nabla f^c(\by)=\nabla_yc(\bar\bx,\by), \quad 
\nabla g^c(\bx)=\nabla_xc(\bx,\bar \by).
\label{ctrans}\ee

\medskip
To derive explicitly the Kantorovich formulation in the present case, we write the cost $c(\bx,\by)$ given by (\ref{newcost}) as
\begin{eqnarray}
c(\bx, \by)&=&\frac 1 2 (x_1^2+x_2^2)+\frac 1 2 (y_1^2+y_2^2)-(x_1y_1+x_2y_2+x_3y_3)
\nonumber \\
&=&
-(\bx,\by)+\frac 1 2 (x_1^2+x_2^2)+\frac 1 2 (y_1^2+y_2^2),
\label{dcost}\end{eqnarray}
where $(\bx,\by)$ denotes the euclidean inner product in $\R^3$.

Then we can write the minimisation problem (\ref{newstablefirst}) in a relaxed form, and (\ref{newstablefirst})  can be formulated as the problem of finding $h\in {\cal H}_0$ and the optimal plan $\gamma\in \Gamma(\sigma_h,\nu)$ minimising
\be\label{newstablerel}
I(\gamma,h)=\int_{\Omega_h\times \Lambda}\left[-(\bx,\by)+\frac 1 2 (x_1^2+x_2^2)+\frac 1 2 (y_1^2+y_2^2)\right]d\gamma(\bx,\by),
\ee
where 

\begin{tabular}{ll}
$\sigma_h\in{\cal H} $ & with ${\cal H}$  given by (\ref{calH});
\\
$\nu$& is a given compactly supported probability measure in $\R^3$;
\\
$\Omega_h$& denotes the support of $\sigma_h$, given by (\ref{cgfreeOmega});
\\
$\Lambda$& denotes the support of $\nu$, as in (\ref{Lambda0});
\\
$ \Gamma(\sigma_h,\nu)$ & denotes the set of probability measure on the product space $\Omega_h\times \Lambda$ which take 
\\
 &$\sigma_h$
 and $\nu$ as marginals. 
\end{tabular}
%

\smallskip
The "marginal" condition  means that  each $\gamma\in\Gamma(\sigma_h,\nu)$ satisfies
$\gamma(A\times \Lambda)=\sigma_h(A)$ for every measurable set $A\subset \Omega$, and  $ \gamma(\Omega\times B)=\nu(B)$ for every measurable set $B\subset \Lambda$. 

\smallskip
The Kantorovich maximisation problem can be stated in terms of $P=\frac 1 2 (x_1^2+x_2^2)-f$ and $R=\frac 1 2 (y_1^2+y_2^2)-g$, where $f$, $g$ are as in (\ref{new26}). The problem is the following: 

\begin{pbm}[Kantorovich formulation]\label{kant}
Find $(P(\bx), R(\by)) \in C(\Omega_H)\times C(\Lambda)$  such that 
\be
P(\bx)+R(\by)\geq (\bx,\by),\quad (\bx,\by)\in\Omega_H\times \Lambda,
\label{PRcond}
\ee
and that maximise
\be
J(P,R)=\left[\int_{\Lambda}\left(\frac 1 2 (y_1^2+y_2^2)-R(\by)\right)d\nu(\by)+\inf_{h\in {\cal M}}\int_{\Omega_H}\left(\frac 1 2 (x_1^2+x_2^2)-P(\bx)\right)d\sigma_h(\bx)\right].
\label{MK}\ee
where ${\cal M}$ is the set of measurable function $h:\Omega_2\to[0,\infty)$. \end{pbm}

Note that, by construction, for any $P$, $R$ as above, and any measurable $h:\Omega_2\to [0,\infty)$, it holds that
\be
J(P,R)\leq I(\gamma,h).
\label{maxmin}\ee

It follows from the general theory (see appendix)  that the solution $(P,R)$ of this maximisation problem is such that $P$ and $R$ are Legendre transforms of each other, namely such that $R=P^*$, $P=R^*$, where  
$$
P^*(\by)=\sup_{\bx\in\Omega_H}((\bx,\by)-P(\bx));\qquad 
R^*(\bx)=\sup_{\by\in\Lambda}((\bx,\by)-R(\by)).
$$
In particular this implies that $P$, $R$ are convex functions, see Lemma \ref{newcconcavechar}.   

\subsubsection{Minimisation with respect to the map  $\mathbf T(\bx)$}\label{freeoptsection}

Suppose that $\sigma\in{\cal P}_{ac}(\R^3)$ is given. Then the maximisation as stated in Problem \ref{kant} is the classical optimal transport problem with respect to a quadratic cost between two probability measures  absolutely continuous  with respect to Lebesgue measure.   Details of the general theory are given in appendix, where we state that there exists  $f$ that maximises (\ref{new26}), and that the optimal transport plan between $\sigma_h(\bx)$ and $\nu(\by)$ is given by $id\times\nabla P$, where $P(\bx)=\frac 1 2 (x_1^2+x_2^2)-f(\bx)$,  see Lemma \ref{newcconcavechar}.

Indeed, in this case there exists an optimal transport map  $\mathbf T$ which solves  (\ref{new3.3}). This map is given by $\mathbf T=\nabla P$, with $P$ the maximiser in (\ref{MK}).   

\begin{theorem}\label{new3.3}
Assume that 
 $\Lambda \subset \mathbb{R}^3$ is a bounded open set.  Let $\optfreem _h\in \optFreem$, $\nu \in \mathcal{P}_{ac}(\Lambda )$ and $c(\cdot , \cdot )$ be defined by (\ref{newcost}).  
  Then, there exist maps $\mathbf T$ 
 and $\mathbf S$, 
 unique $\optfreem _h-$a.e. and $\nu -$a.e. respectively, and a convex function $P$ such that
\begin{enumerate}[label=(\textit{\roman{*}}), ref=\textit{(\roman{*})}]
	\item $\mathbf T = \nabla P$ is optimal in the transport of $\optfreem _h$ to $\nu $ with cost $c(\mathbf x, \mathbf y)$,
	\item $\mathbf S = \nabla P^*$ is optimal in the transport of $\nu $ to $\optfreem _h$ with cost $\tilde{c}(\mathbf y, \mathbf x) = c(\mathbf x, \mathbf y)$, where $P^*$ is defined in (\ref{legendredef},
\item $\mathbf S$ and $\mathbf T$ are inverses, i.e. $\mathbf S\circ \mathbf T(\mathbf x) = \mathbf x$ for $\optfreem _h-$a.e. $\mathbf x$ and $\mathbf T\circ \mathbf S(\mathbf y) = \mathbf y$ for $\nu -$a.e. $\mathbf y$.
\end{enumerate}
\end{theorem}

%
We also have the following stability result.

\begin{lemma}\label{optnew3.4}
Assume that $\Lambda \subset \mathbb{R}^3$ is a bounded open set.  Let $c(\cdot , \cdot )$ be defined by (\ref{newcost}).  Define $E_\nu (\cdot )$ by (\ref{newenergyminfirst}).  Let $\sigma_{h_n}, \, \sigma_h \in {\cal H} $ and $\nu _n,\, \nu\in \mathcal{P}_{ac}(\Lambda )$ with $\sigma_{h _n}$ converging to $\sigma_h$ in ${\cal H}$ and $\nu _n$ converging narrowly to $\nu $ as $n \to \infty $.  Then $E_{\nu _n}(h _n) \to E_{\nu}(h )$ as $n \to \infty $, i.e.
\begin{equation}\label{newconverge}
\inf _{\overline{\mathbf T} \push \optfreem _{h_n} = \nu _n}\int_{\mathbb{R}^3} \! c(\mathbf x, \overline{\mathbf T}(\mathbf x))\optfreem _{h_n} \, d\mathbf x \longrightarrow \inf _{\overline{\mathbf T} \push \optfreem _h = \nu }\int_{\mathbb{R}^3} \! c(\mathbf x, \overline{\mathbf T}(\mathbf x))\optfreem _h\, d\mathbf x, \qquad \textrm{ as }n \to \infty .
\end{equation}
\end{lemma}

\begin{proof}
The proof is standard in optimal transport theory, and similar to the analogous proof in \cite{maroofi}, since (\ref{htop}) implies narrow convergence of $\sigma_{h_n}$ to $\sigma_h$.

\end{proof}


  \subsubsection{Minimisation with respect to the function  $h(x_1,x_2)$}\label{sectionfreeenergymin}

Assume that $P(\bx)$ is a convex function such that, for fixed $(x_1,x_2)$, the function  $\tilde P(x_3)=\frac 1 2 (x_1^2+x_2^2)-P(\bx)$ is  nonzero on a subset $I\subset \R$ of positive Lebesgue measure and satisfies
\be
\frac {d  \tilde P}{d x_3}=-\frac {\partial P}{\partial x_3}\geq 0.
\label{onesign}\ee

Given $\nu \in \mathcal{P}_{ac}(\Lambda)$,  we aim to prove that there exists a unique measurable function $h(x_1,x_2):\Omega_2\to [0,H)$ which is a  minimiser for the second term  in (\ref{MK}). 

Define
\be\label{Pip}
\Pi_P(x_1,x_2,s)=\int_0^s \left[\frac 1 2 (x_1^2+x_2^2)-P(x_1,x_2,x_3)\right] dx_3.
\ee



%


The function $\Pi_P(x_1,x_2,s)$ admits a minimum in $s$ (by continuity), and it follows from our assumption on $\tilde P(x_3)$ that the point $s^*$ where the  minimum is attained is unique.   Indeed, if $s_*\neq 0$ is such a point,  then from the condition $\frac{\partial \Pi_P}{\partial s}(s_*)=0$ we obtain
\be
 P(x_1,x_2,s_*)=\frac 1 2 (x_1^2+x_2^2), 
\label{mincond}\ee
i.e.  $\tilde P(s_*)=0$ if $s_*\neq 0$ is a minimiser. Integrating by parts the integral in (\ref{Pip}), we obtain
$$
\int_0^s \tilde P(x_3)dx_3=s\tilde P(s)-\int_0^s x_3\frac {d \tilde P}{dx_3}.$$
Hence if  $s=s_*$ is a minimiser,  it follows from (\ref{mincond}) that  the first term on the right hand side vanishes. Hence if two such nonzero points exist, say $s_*^1$ and  $s_*^2$, it must be 
$
\int_{s_*^1}^{s_*^2}x_3\frac {d \tilde P}{dx_3}=0.
$
Given our assumption that  the integrand is of one sign, this implies $s_*^1=s_*^2$.

Now suppose $s^*=0$ is a point of  minimum, hence  that   $\Pi_P(x_1,x_2,s)\geq 0$ for all $s\in[0,H)$.  Since the integrand must then be nonnegative, a point $s$ can minimise $\Pi_P(x_1,x_2,s)$ only if (\ref{mincond}) holds at $s^*=s$. 

Hence the minimiser of  $\Pi_P(x_1,x_2,s)$ is unique.

\begin{definition}\label{fhdef}
Define $$h(x_1,x_2)=s_*$$ where $s_*$ is the unique minimiser of $\Pi_P$ given by (\ref{Pip}).
\end{definition}

Following the argument presented in  \cite{sedjro}
for a more difficult situation, one can  show  that the $h$ given by Definition \ref{fhdef} is well defined and that $h(x_1,x_2)\in L^1\cap L^2(\Omega_2)$.

\begin{remark}
As discussed in the previous section, the optimal transport map between $\sigma_h$ and the given density $\nu$ exists and takes the form  $\mathbf T=\nabla P_h$, where $P_h$ is convex. Moreover, $P_h$ is related to the physical pressure $p$ by the relation (see equation (\ref{Pp}) below)
$$
P_h= p + \frac{1}{2}(x_1^2 + x_2^2).
$$
It follows that
$$
\frac {\partial P_h}{\partial x_3}=\frac {\partial p}{\partial x_3} =-\rho\leq 0.
$$
%
Hence $P_h$ satisfies the assumption (\ref{onesign}).

In addition, since the integrand $x_3\rho(\bx)$ is nonnegative, the value $s=0$ cannot be a minimum unless $\rho(\bx)=0$ a.e. (with respect to $x_3$), a case we exclude, see  (\ref{densityest}). 
\end{remark}

\subsubsection{The minimisation result}


In this section, using the results of the minimisation separately in $\mathbf T$ and $h$, we show that there exist a pair of convex functions $(P(\bx),R(\by))$ that maximise (\ref{MK}), and such that  $P(\bx)$ satisfies the additional requirement to be a monotonic function of the variable $x_3$. 
Then we are able to show that there exists a unique minimiser $(\gamma,h)$ of (\ref{newstablerel}), where $\gamma=Id\times \nabla P$ and $h\in {\cal H}_0$, where ${\cal H}_0$ is given in (\ref{H0}).

\begin{proposition}
Let $\nu\in {\cal P}_{ac}(\Lambda)$ be  given.
\begin{itemize}
\item[(i)]
For all $h\in{\cal M}$, there exist  a pair $(P(\bx), R(\by))$, as in Problem (\ref{kant}), that  maximise (\ref{MK}), and such that $P(\bx)$ satisfies condition (\ref{onesign}).
\item[(ii)]
Assume that $(\gamma,h)$ is an arbitrary  pair with $h\in {\cal H}_0$ and  $\gamma\in\Gamma(\sigma_h,\nu)$, as in (\ref{newstablerel}). 

Then
$
I(\gamma,h)\geq J(P,R)$ for all $(P,R)$ as in (i).

The equality holds if and only if $h(x_1,x_2)\in {\cal H}_0$ minimizes $\Pi_P(x_1,x_2,\cdot)$ a.e., and $id\times \nabla P$ is the optimal transport plan of $\sigma_h$ to $\nu$.

\end{itemize}
\label{sed3.4}
\end{proposition}

\begin{proof}

(i): We sketch the proof, which is based on the analogous proof of part (i) of Proposition 3.4 in \cite{sedjro}, but simpler in our case.

The set defined by the conditions in (\ref{kant}) is non-empty. Indeed, define
$$
c_0=sup_{\Omega_H\times \Lambda}  (\bx,\by),\quad P_0(\bx)=c_0,\quad R_0(\by)=0.
$$
Then $(P_0,R_0)$ satisfy all conditions in Problem (\ref{kant}).

Now consider a maximising sequence $(P_n,R_n)$  for $J(P,R)$. By the double convexification trick, it can be assumed without loss of generality that this sequence is convex, and it can be shown \cite{sedjro} that it converges uniformly to a pair $(P,R)$ of convex functions satisfying the conditions in Problem (\ref{kant}), and in addition such that $P$ satisfies (\ref{onesign}). The latter is a consequence of the fact that the support $\Lambda$ of $\nu$  has the property (\ref{Lambda0}).
Then by standard stability results, $J(P,R)$ is a maximum of the functional.

\smallskip
(ii): Let $(\gamma,h_0)$ be an arbitrary pair with $h_0\in {\cal H}_0$ and $\gamma\in \Gamma(\sigma_{h_0},\nu)$, and let $(P,R)$ be as in part (i).  Then as already observed (see (\ref{maxmin})) we have using that $P(\bx)+R(\by)\geq (\bx,\by)$
$$
J(P,R)=\left[\int_{\Lambda}\left(\frac 1 2 (y_1^2+y_2^2)-R(\by)\right)d\nu(\by)+\inf_{h\in {\cal M}}\int_{\Omega_H}\left(\frac 1 2 (x_1^2+x_2^2)-P(\bx)\right)d\sigma_h(\bx)\right]\leq
$$
$$
\left[\int_{\Lambda}\left(\frac 1 2 (y_1^2+y_2^2)-R(\by)\right)d\nu(\by)+\int_{\Omega_{h_0}}\left(\frac 1 2 (x_1^2+x_2^2)-P(\bx)\right)d\bx\right]\leq
I(\gamma,h_0).
$$
Equality holds if and only if it holds that
$$
\inf_{h\in {\cal M}}\int_{\Omega_H}\left(\frac 1 2 (x_1^2+x_2^2)-P(\bx)\right)d\sigma_h(\bx)=\int_{\Omega_{h_0}}\left(\frac 1 2 (x_1^2+x_2^2)-P(\bx)\right)d\bx
$$
i.e. if $h_0(x_1,x_2)$ is the minimiser of (\ref{Pip}), and if it holds that
$$
\left[\int_{\Lambda}\left(\frac 1 2 (y_1^2+y_2^2)-R(\by)\right)d\nu(\by)+\int_{\Omega_{h_0}}\left(\frac 1 2 (x_1^2+x_2^2)-P(\bx)\right)\bx\right]=
I(\gamma,h_0).
$$
Using the fact that all measures involved are absolutely continuous with respect to Lebesgue measure, the second condition implies that $\by=\nabla P (\bx)$ a.e., and hence the map $\gamma$ is of the form $id\times \nabla P$.

\end{proof}

\begin{corollary}\label{cormin}
There exists a unique minimiser $(\gamma,h)$  of the functional  $I(\gamma,h)$ given by (\ref{newstablerel}), with $h\in {\cal H}_0$ and $\gamma\in\Gamma(\sigma_h,\nu)$. In addition, if $(P_0,R_0)$ as in (i) maximises $J(P,R)$, then $J(P_0,R_0)=I(\gamma,h)$ and $(h,T=\nabla P)$ is the unique minimiser of (\ref{newstablefirst}).
\end{corollary}

\subsubsection{Properties of the energy minimiser}

In this section we use the notation $E_\nu(h)\equiv E_\nu(\sigma_h)$.

\begin{theorem}\label{new4.2}
Assume that $\Lambda \subset \mathbb{R}^3$ is a bounded open set and let $\nu \in \mathcal{P}_{ac}(\Lambda )$.  Let $\Freem $ be defined by (\ref{calH}).  Let $\freemin$ correspond to the unique minimiser of $E_\nu (\cdot )$ in $\Freem $.  Denote by $\bar {\mathbf T}$ the optimal map in the transport of $\sigma_{\freemin} $ to $\nu $ with cost function $c(\mathbf x, \mathbf y)$ defined as in (\ref{newcost}).  Then $\bar {\mathbf T} = \nabla P$, where
\begin{equation}\label{Pp}
P = p + \frac{1}{2}(x_1^2 + x_2^2).
\end{equation}
Moreover,
\begin{equation}\label{psob1}
   p(t, \cdot ) \in W^{1, \infty }(\Omega _2 \times [0, \freemin ]), \quad \bar h(t,\cdot) \in W^{1,\infty }(\Omega _2),\qquad \textrm{ for all } t \in [0, \tau ).
  \end{equation}
\end{theorem}

\begin{proof}
Let $\xi $ be a smooth, compactly supported  vector field in $\Omega _2 \times (0, \infty)$ such that $\nabla\cdot \xi=0$.
Consider the one parameter family of measure-preserving diffeomorphisms $\{\mathbf R(s,\bx)\}$ given by 
\[ \frac{\partial }{\partial s}\mathbf R(s, \mathbf x) = \xi (\mathbf R(s, \mathbf x)), \qquad \mathbf R(0, \mathbf x) = \mathbf x.\]
For $s > 0$, let $\optfreem _s \in \optFreem $ be given by $\optfreem _s : = \mathbf R \textrm{\#} \sigma_{\bar h} $. 
Since the flow corresponding to $\xi$ is smooth and incompressible,   we can assume that for $s$ sufficiently small it transports the initial minimising profile $\bar h$ to some perturbed profile $\tilde h_s$, which however may be multivalued.  Using Lemma \ref{lemmasingleval}, we can find a corresponding single-valued $h_s\in 
L^1\cap L^2(\Omega_2)$  with $\|h_s\|_1=1$,  whose corresponding energy is lower than the energy associated with $\tilde h_s$.  Hence we can assume in the argument  that $\sigma_s=\sigma_{h_s}$.  


 

Since $\freemin $ corresponds to the minimiser for $E_\nu$, we have
\begin{align*}
0 \leqslant E_\nu(\freem _s) - E_\nu(\freemin ) & = \inf _{\mathbf T \textrm{\#} \optfreem _s = \nu}\int_{\mathbb{R}^3} \! c(\mathbf x, {\mathbf T}(\mathbf x))\optfreem _s(\mathbf x) \, d\mathbf x - \inf _{\mathbf T \textrm{\#} \optfreemin = \nu }\int_{\mathbb{R}^3} \! c(\mathbf x, {\mathbf T}(\mathbf x))\optfreemin (\mathbf x) \, d\mathbf x\\
& = \int_{\mathbb{R}^3} \! c(\mathbf x, \mathbf T_s(\mathbf x))\optfreem _s (\mathbf x)\, d\mathbf x - \int_{\mathbb{R}^3} \! c(\mathbf x, \bar {\mathbf T}(\mathbf x))\optfreemin (\mathbf x)\, d\mathbf x,
\end{align*}
where the existence and a.e. uniqueness of optimal maps $ \mathbf T_s$ and $\bar {\mathbf T}$ follow from  \ref{cormin}.  Since $(\bar {\mathbf T} \circ \mathbf R^{-1})\textrm{\#} \optfreem _{s} = \nu $, and $\mathbf R(0, \mathbf x) = \mathbf x$, we have
\begin{align*}
0 \leqslant \lim _{s \to 0}\frac{E_\nu(\freem _s) - E_\nu(\freemin )}{s} &\leqslant \lim _{s \to 0}\frac{\int_{\mathbb{R}^3} \! c(\mathbf x, \bar {\mathbf T}(\mathbf R^{-1}(s,\mathbf x))) \optfreem _s(\mathbf x)\, d\mathbf x - \int_{\mathbb{R}^3} \! c(\mathbf x, \bar {\mathbf T}(\mathbf x))\optfreemin (\mathbf x) \, d\mathbf x}{s} \\
& = \lim _{s \to 0}\frac{\int_{\mathbb{R}^3} \! [c(\mathbf R(s,\mathbf x), \bar {\mathbf T}(\mathbf x)) -  c(\mathbf x, \bar {\mathbf T}(\mathbf x))] \optfreemin (\mathbf x)\, d\mathbf x}{s}\\
& = \lim _{s \to 0}\frac{\int_{\Omega _2}\int_0^{\freemin } \! c(\mathbf R(s,\mathbf x), \bar {\mathbf T}(\mathbf x)) -  c(\mathbf x, \bar {\mathbf T}(\mathbf x)) \, d\mathbf x}{s}\\
& = \int_{\Omega _2}\int_0^{\freemin } \! \nabla c(\mathbf x, \bar {\mathbf T}(\mathbf x)) \cdot \xi (\mathbf x) \, d\mathbf x .
\end{align*}
Using the assumption that  $\xi $ is an arbitrary vector field in the class chosen, this inequality also holds for $-\xi $. Therefore (using that 
$\xi$ is divergence free) we can deduce that
\[ \int_{\Omega _2}\int_0^{\freemin } \! \nabla \cdot (c(\mathbf x, \bar {\mathbf T}(\mathbf x)) \xi (\mathbf x)) \, d\mathbf x = 0.  \]
We now want to conclude that $\nabla c(\bx, \bar {\mathbf T}(\mathbf x))=-\nabla p(\bx)$, in the weak sense. 

Since $\bar h$ is not necessarily smooth, we cannot use the Gauss-Green theorem. However we note that $\Omega_h$ is a finite perimeter set, i.e. $\sigma_{\bar h}$ is of bounded variation. 
Therefore we can use  the generalisation of the divergence theorem due to De Giorgi (see for example \cite{ambperimetri}) to conclude 
\begin{equation} 0=\int_{\Omega _h} \! \nabla \cdot (c(\mathbf x, \bar {\mathbf T}(\mathbf x)) \xi (\mathbf x)) \optfreemin (\mathbf x) \,  d\mathbf x = \int_{\partial \Omega ^*_{\bar h}}((\xi \cdot \mathbf n) c(\mathbf x, \bar {\mathbf T}(\mathbf x)) |_{x_3 = \freemin }) \, dx_1dx_2.  \label{deg}\end{equation}
In this formula, $\partial \Omega ^*_{\bar h}$ denotes the reduced boundary (in the sense of De Giorgi) of $\Omega_h$. The only nonzero boundary terms  are the ones arising from the portion of reduced boundary which is a subset of $\{x_3=\bar h(x_1,x_2)\}$.  Given the boundary condition $p = 0$ when $x_3 = \freemin $, we conclude that the identity (\ref{deg})  will hold for arbitrary $\xi $ if  
the identity
$ \nabla c(\mathbf x, \bar {\mathbf T}(\mathbf x)) = -\nabla p(\mathbf x)$ holds in the weak sense. Using the fact that $\bar{\mathbf T}=\nabla P$ with $P$ convex (see Theorem \ref{new3.3}), the properties of $c$-transforms (see (\ref{ctrans}), and the argument for the validity of identity (\ref{mincond}), this implies that
  \begin{equation}\label{psob}
  p \in W^{1, \infty }(\Omega ),\quad \forall t \in [0, \tau ),
  \end{equation}
and that 
 $$P = p + \frac{1}{2}(x_1^2 + x_2^2) \qquad a.e. \;in \; \Omega _2 \times [0, \freemin ].
 $$

We now note that  since the pressure $p(x_1,x_2,x_3)$ satisfies, for each fixed time $t<\tau$,
\be
\left\{\begin{array}{ll}
\frac{\partial p}{\partial x_3}=-\rho&
\\
p=p_h& x_3=h(x_1,x_2)
\end{array}\right. \label{condt1}\ee
it is possible to establish a relation between $p$ and $h$. Indeed, let $\tilde p(x_1,x_2)=p(x_1,x_2,h(x_1,x_2))$. By the given boundary conditions, the function $\tilde p$ is constant, hence
$$
\nabla_2 \tilde p=\nabla_2 p+\frac {\partial p}{\partial x_3}\nabla_2 h=0,
$$
where $\nabla_2 =\left(\frac{\partial}{\partial x_1}, \frac{\partial}{\partial x_2}\right)$ denotes the two-dimensional gradient. Using the condition (\ref{condt1}), we find
$$
\nabla_2 p=\rho \nabla_2 h, \quad (x_1,x_2)\in\Omega_2.
$$


Hence  control of $h$  follows from the above estimates on $p$ and $\rho$, since $D_t\rho=0$ implies that  $\rho$ is bounded by its initial values. We can therefore assert that the unique solution of the minimisation problem satisfies additionally the property $h\in W^{1,\infty}(\Omega_2)$. 

\end{proof}
\begin{lemma}\label{new4.3}
Assume that $\Lambda \subset \mathbb{R}^3$ is a bounded open set.  Let $\nu _n ,\, \nu \in \mathcal{P}_{ac}(\Lambda )$ with $\nu _n$ converging narrowly to $\nu $ as $n \to \infty $.  Let $\Freem $ be defined by (\ref{calH}) and let $E_\nu (\cdot )$ be defined by (\ref{newstablefirst}).  For each $n$, let $\freemin _n\in {\cal H}_0$ correspond to the minimiser $\sigma_{\bar h_n}\in{\cal H}$ of $E_{\nu _n}(\cdot )$ and let $\freemin \in {\cal H}_0$ correspond to the minimiser $\sigma_{\bar h}\in {\cal H}$ of $E_{\nu }(\cdot )$.  Then, as $n \to \infty $, $\freemin _n$ converge  to $\freemin $  in the weak $L^1$ topology.
\end{lemma}

\begin{proof}


For each $n$, $\sigma_{\bar h _n}$ minimises $E_{\nu _n}(\cdot )$ over $\Freem $ so that
\begin{equation}\label{contradiction}
E_{\nu _n}(\freemin _n) \leqslant E_{\nu _n}(\freem )\quad \forall h\in {\cal H}_0.
\end{equation}

By the compactness of $\Freem $, there exists $\tilde{\freem } \in {\cal H}_0 $ such that, up to a subsequence that we label $\freemin _n$ again, the $\sigma_{\bar h _n}$ converge in ${\cal H}$, and hence  narrowly, to $\sigma_{\tilde h}$. 
  
  We now show that $\tilde{\freem}$ minimises $E_\nu(\cdot )$.  From Lemma \ref{optnew3.4}, we have that $E_{\nu _n}(\freemin _n) \to E_\nu (\tilde{\freem})$ as $n \to \infty $, and that $E_{\nu _n} (\freemin ) \to E_\nu (\freemin )$ as $n \to \infty $.  

Since $\sigma_{\bar h} $ minimises (\ref{newenergyminfirst}) over $\Freem $, we know that $E_\nu (\freemin ) \leqslant E_\nu (\tilde{\freem})$.  Assume that $E_\nu (\freemin ) < E_\nu (\tilde{\freem})$.  Then, for $n$ large enough, $E_{\nu _n}(\freemin ) < E_{\nu _n}(\freemin _n)$.  This contradicts (\ref{contradiction}).  Thus, we obtain $E_\nu (\tilde{\freem}) = E_\nu (\freemin ) \leqslant E_\nu (\freem )$ for all $\freem \in \Freem $.  
Hence, we have that $\tilde{\freem}$ is a minimiser of $E_\nu (\cdot )$.  Since the minimiser of $E_\nu (\cdot )$ is unique, we conclude that $\tilde{\freem} = \freemin $.

%

\end{proof}

\subsection{Dual space existence result}\label{newmain}

In this section we  prove our main result, namely Theorem \ref{new5.5}. 

We will make use of the theory of Hamiltonian ODE of \cite{gangbo},  summarised in the Appendix. 
Here we give  the main definition and list the properties of the Hamiltonian that allow us to invoke that theory of Hamiltonian flows in the present context.

The concept of Hamiltonian ODEs is rigorously defined in appendix. In brief, and in the present context,  a Hamiltonian flow is the solution $\nu(t)\in{\cal P}_{ac}^2$ of the following problem:
given an initial probability density $\nu_0\in\mathcal{P}_{ac}^2$ and a Hamiltonian $H: \mathcal{P}_{ac}^2\to \R$,  find $\nu(t)\in{\cal P}_{ac}^2$  that coincides with $\nu_0$ at time $t=0$ and satisfying
$$
\partial _t\nu(t)+\nabla \cdot (J\partial_0 H(\nu(t))\nu(t))=0.$$
Here $\partial_0 H$ denotes in general the element of minimal $L^2$ norm in the superdifferential of $H$ (see Appendix).
The space $\mathcal{P}_{ac}^2$ is consider as a metric space with the metric given by the Wasserstein distance $W_2$ \cite{ambgangbo}. 

The strategy of the proof of Theorem \ref{new5.5} is to show that when one consider as Hamiltonian the dual energy, it is posible to find a corresponding Hamiltonian flow, and moreover that the velocity $J\partial_0 H(\nu(t))$ coincides with the dual velocity $\mathbf w$.

The  three conditions (H1), (H2), (H3)  on the Hamiltonian $H$  that guarantee the existence of an Hamiltonian flow are the following: \\*[3mm]
(H1) \textit{ There exist constants } $C_0 \in (0, \infty )$, $R_0 \in (0, \infty ]$ \textit{such that, for all }$\nu \in \mathcal{P}_{ac}^2(\mathbb{R}^3)$ \textit{with} $W_2(\nu , \nu _0) < R_0$, \textit{ we have that the superdifferential is not empty }  \textit{ and } $\mathbf v = \partial _0 H(\nu )$ \textit{ satisfies } $|\mathbf v(\mathbf y)| \leqslant C_0(1 + |\mathbf y|)$ \textit{ for } $\nu -$a.e. $\mathbf y \in \mathbb{R}^3$.\\*[3mm]
(H2) \textit{ If } $\nu ,\, \nu _n \in \mathcal{P}_{ac}^2(\mathbb{R}^3)$, $\sup _n W_2(\nu _n, \nu_0)<R_0$ \textit{ and } $\nu _n \rightarrow \nu $ \textit{ narrowly, then there exists a subsequence } $n(k)$ \textit{ and functions } $\mathbf v_k$, $\mathbf v$ \textit{ such that } $\mathbf v_k = \partial _0 H(\nu _{n(k)})$ $\nu _{n(k)}-$a.e., $\mathbf v = \partial _0 H(\nu )$ $\nu -$a.e. \textit{ and } $\mathbf v_k \rightarrow \mathbf v$ a.e. \textit{ in } $\mathbb{R}^3$ \textit{ as } $k \rightarrow \infty $.\\*[3mm]

Condition (H1) essentially requires that the velocity's growth is controlled and bounded on every bounded domain,
while condition (H2) is a continuity assumption.

To ensure the constancy of $H$ along the solutions of the Hamiltonian system we consider also:\\*[3mm]
(H3) $H:\mathcal{P}_{ac}^2(\mathbb{R}^3) \rightarrow (-\infty , \infty ]$ \textit{ is proper, upper semi-continuous and }$\lambda -$\textit{concave for some }$\lambda \in \mathbb{R}$.\\*[3mm]

 For $\nu \in \mathcal{P}^2_{ac}(\Lambda )$,  we define the Hamiltonian $H(\nu)$ as given by dual geostrophic energy:
 \be
 H(\nu):={\cal E}(t,\nu),\quad {\cal E}(t,\nu) \;given\;by \; (\ref{newstablefirst}).
 \label{Hamilt}\ee
 
 The main result of \cite{ambgangbo}, Theorem (\ref{ag6.6}),   states that, if (H1), (H2) hold for $H(\nu )$, then at least for some time there exists an absolutely continuous Hamiltonian flow $\nu _{(t)} \in \mathcal{P}_{ac}(\Lambda )$ satisfying (\ref{33}) such that $t \mapsto \nu _{(t)}$ is Lipschitz, and $\Lambda\subset \R^3$ is a bounded open set.  If in addition (H3) holds, then $t \mapsto H(\nu _{(t)})$ is constant.

\medskip
We begin with showing that the Hamiltonian is superdifferentiable.
\begin{proposition}\label{newHok}
Let $\Lambda \subset \mathbb{R}^3$ be an open bounded set.  Let the Hamiltonian $H(\nu)$ on $\mathcal{P}^2_{ac}(\Lambda )$ be defined by (\ref{Hamilt}). Then $H$ is superdifferentiable, upper semi-continuous and $(-2)-$concave.
\end{proposition}


\begin{proof}

Given $\nu \in \mathcal{P}^2_{ac}(\Lambda )$, denote by $\freemin $ the minimiser in (\ref{newstablefirst}). 
 The existence and a.e. uniqueness of this minimiser follows from Corollary  \ref{cormin}.  For any $\nut \in \mathcal{P}^2_{ac}(\Lambda)$ we have
\begin{equation*}
H(\nut) = \inf _{\sigma_h \in \Freem}  {\cal E}(\nut, \optfreem_h )   \leqslant {\cal E}(\nut, \sigma_{\bar h} ).
\end{equation*}

Let $\mathbf R^{\nut}_{\nu }$ be the (unique)  optimal transport map from $\nu $ to $\nut$ with respect to the usual quadratic cost.

Consider the transport with respect to the cost function $\tilde{c}(\mathbf y, \mathbf x) = c(\mathbf x, \mathbf y)$ given by (\ref{newcost}).   Let $\mathbf S_\nu ^{\optfreemin} $ be the optimal map in the transport of $\nu $ to $\optfreemin $ and let $\mathbf S_{\nut}^{\optfreemin} $ be the optimal map in the transport of $\nut$ to $\optfreemin $.  Therefore, we have
\[ \inf _{\mathbf S \textrm{\#} \nu = \optfreemin  } \int_\Lambda \! \tilde{c}(\mathbf y, \mathbf S (\mathbf y)) \nu (\mathbf y)  \, d\mathbf y =  \int_\Lambda \! \tilde{c}(\mathbf y, \mathbf S_\nu ^{\optfreemin}(\mathbf y)) \nu (\mathbf y)  \, d\mathbf y \]
and
\[\inf _{\mathbf S \textrm{\#} \nut= \optfreemin } \int_\Lambda \! \tilde{c}(\mathbf y, \mathbf S (\mathbf y)) \nut(\mathbf y)  \, d\mathbf y = \int_\Lambda \! \tilde{c}(\mathbf y, \mathbf S_{\nut}^{\optfreemin } (\mathbf y)) \nut(\mathbf y)  \, d\mathbf y .\]
The existence of $\mathbf S_{\nu }^{\optfreemin}$ and $\mathbf S_{\nut}^{\optfreemin} $ follows from Theorem \ref{new3.3}.  Note that, since $(\mathbf S_{\nu }^{\optfreemin } \circ (\mathbf R^{\nut}_{\nu })^{-1}) \textrm{\#} \nut= \optfreemin  $ and since $\mathbf S_{\nut}^{\optfreemin} $ is optimal in the transport of $\nut$ to $\optfreemin $, we have
\[ \int_\Lambda \! \tilde{c}(\mathbf y, \mathbf S_{\nut}^{\optfreemin} (\mathbf y)) \nut(\mathbf y)  \, d\mathbf y \leqslant \int_\Lambda \! \tilde{c}(\mathbf y, \mathbf S_{\nu }^{\optfreemin} \circ (\mathbf R^{\nut}_{\nu })^{-1} (\mathbf y)) \nut(\mathbf y)  \, d\mathbf y. \]
It follows that
\begin{align*}
H(\nut) - H(\nu ) &\leqslant {\cal E}(\nut, \optfreemin )  - {\cal E}(\nu , \optfreemin ) \\
&= \int_\Lambda \! \tilde{c}(\mathbf y, \mathbf S_{\nut}^{\optfreemin} (\mathbf y)) \nut(\mathbf y)  \, d\mathbf y - \int_\Lambda \! \tilde{c}(\mathbf y, \mathbf S_{\nu }^{\optfreemin } (\mathbf y)) \nu (\mathbf y)  \, d\mathbf y \\
&\leqslant \int_\Lambda \! \tilde{c}(\mathbf y, \mathbf S_{\nu }^{\optfreemin } \circ (\mathbf R^{\nut}_{\nu })^{-1} (\mathbf y)) \nut(\mathbf y)  \, d\mathbf y - \int_\Lambda \! \tilde{c}(\mathbf y, \mathbf S_{\nu }^{\optfreemin } (\mathbf y)) \nu (\mathbf y)  \, d\mathbf y \\
&= \int_\Lambda \! \tilde{c}(\mathbf R^{\nut}_{\nu }(\mathbf y), \mathbf S_{\nu }^{\optfreemin } (\mathbf y)) \nu (\mathbf y)  \, d\mathbf y - \int_\Lambda \! \tilde{c}(\mathbf y, \mathbf S_{\nu }^{\optfreemin } (\mathbf y)) \nu (\mathbf y)  \, d\mathbf y \\
&= \int_\Lambda \! \bigg[ \tilde{c}(\mathbf R^{\nut}_{\nu }(\mathbf y), \mathbf S_{\nu }^{\optfreemin } (\mathbf y)) - \tilde{c}(\mathbf y, \mathbf S_{\nu }^{\optfreemin } (\mathbf y)) \bigg]\nu (\mathbf y)  \, d\mathbf y, \\
& = \int_\Lambda \!\nabla \tilde{c}(\mathbf y, \mathbf S_{\nu }^{\optfreemin } (\mathbf y))\cdot [\mathbf R^{\nut}_{\nu }(\mathbf y) - \mathbf y] \nu (\mathbf y)  \, d\mathbf y + o(W_2(\nu , \nut)). 
\end{align*}
\begin{equation}\label{newotherfloaty}
\end{equation}

Hence, using Definition \ref{superdifferential}, we conclude that $\nabla \tilde{c}(\mathbf y, \mathbf S_{\nu }^{\optfreemin } (\mathbf y)) \in \partial H(\nu )$. 
 Thus, $\partial H(\nu )$ is non-empty, $H$ is superdifferentiable and we can use Proposition \ref{vill10.12} to conclude that 
\begin{equation}\label{newsemiconcave}
H \textrm{ is  } (-2)-\textrm{concave}.
\end{equation}
Also, from the 
 continuity of ${\cal E}(\cdot, \cdot )$ (see Lemma \ref{optnew3.4}) and the narrow convergence to $\sigma_h $ as the minimiser of (\ref{newstablefirst}), we have that 
\begin{equation}\label{newusc}
H \textrm{ is upper semi-continuous.}
\end{equation} 
From (\ref{newsemiconcave}) and (\ref{newusc}), we have that (H3) holds.
\end{proof}

\begin{proposition}\label{newmainalt}
Let $1 < r < \infty $ and $\nu _0 \in L^r (\Lambda _0 )$ be an initial potential density with support in $\Lambda _0$, where $\Lambda _0 $ is a bounded open set in $\mathbb{R} ^3$. 
Let $\Lambda \subset \mathbb{R}^3$ be an open bounded set.  Let the Hamiltonian $H=E(t,\nu)$ be defined by (\ref{newstablefirst}).  Then, there exists a Hamiltonian flow $\nu _{(t)} \in \mathcal{P}^2_{ac}(\Lambda )$ and constant $\tau >0$ 
 such that
\[ \frac{d}{dt}\nu _{(t)} + \nabla \cdot (\tilde{J}( \mathbf v_{(t)}) \nu _{(t)}) = 0, \qquad \nu _{(0)} = \nu _0, \qquad t \in (0, \tau )\]
where $\tilde{J}( \mathbf v_{(t)})=\mathbf w$ a.e. in $[0,\tau ]$, and for all $t<\tau$, supp$(\nu_{(t)}\subset \Lambda$ where $\Lambda$ is a bounded open set in $\R^3$.
\end{proposition}

\begin{proof}


We compute $\partial _0H(\nu )$ (as defined in Definition \ref{superdifferential}) explicitly to show that the conditions required to apply Theorem \ref{ag6.6} hold.  From the definition of $\tilde{J}$ in (\ref{tildeJ}), velocity fields transporting $\nu $ will have vanishing components in the $y_3$ direction so that we need only consider variations of $\nu $ in the $(y_1, y_2)-$directions.  
  Thus, to characterise the elements of $\partial H(\nu )$, we let $\tilde{\varphi} \in C_c^\infty (\mathbb{R}^2)$ 
 and define $\varphi (y_1,y_2,y_3):=\tilde{\varphi}(y_1,y_2)$ for all $\mathbf y \in \mathbb{R}^3$.  We then set


  
 \[ \mathbf g_s(\mathbf y) = ((g_s)_1(\mathbf y), (g_s)_2(\mathbf y), (g_s)_3(\mathbf y)) = \mathbf {y} + s\nabla \varphi (\mathbf y).\]
Note that $(g_s)_3(\mathbf y) = y_3$ and, for $|s|$ sufficiently small, $\mathbf g_s$ is the gradient of a convex function, since $\mathbf g_s(\mathbf y) = \nabla (\frac{1}{2}\mathbf y^2 + s\varphi )$.   Define $\nu _s = \mathbf g_s \textrm{\#} \nu $.   Denote by $\freemin _s$ the minimiser in
 \[ H(\nu _s)= \inf _{\sigma_h\in \Freem }  {\cal E}(\nu _s, \sigma _h ) ,\]
 and let $\sigma _s : = \sigma_{\overline h_s}$.  The existence and uniqueness of the minimiser $\freemin _s$ follows from the minimisation result in Corollary \ref{cormin}.  Let $\xi \in \partial H(\nu )$.  
 Combining the $(-2)-$concavity of $H$ and (\ref{newusc}) with Proposition \ref{characterise}, we obtain
 \begin{equation}\label{newinequality}
  H(\nu _s) - H(\nu ) - \int_\Lambda \!  \xi (\mathbf y) \cdot (\mathbf R^{\nu _s}_{\nu }(\mathbf y) - \mathbf {y} )  \nu (\mathbf y) \, d\mathbf y + W_2^2(\nu , \nu _s) \leqslant 0.
  \end{equation}
Since, for $|s|$ sufficiently small, $\mathbf g_s$ is the gradient of a convex function, we conclude that
 \begin{equation*}
 W_2^2(\nu , \nu _s) = \int_\Lambda \! |\mathbf{y} - \mathbf R^{\nu _s}_{\nu }(\mathbf y)|^2 \nu (\mathbf y) \, d\mathbf y = \int_\Lambda \! |\mathbf{y} - \mathbf g_s(\mathbf y)|^2 \nu (\mathbf y) \, d\mathbf y = s^2\int_\Lambda \! |\nabla \varphi (\mathbf y)|^2 \nu (\mathbf y) \, d\mathbf y 
 \end{equation*}
 and
 \[ \int_\Lambda \!  \xi (\mathbf y) \cdot ( \mathbf R^{\nu _s}_{\nu } (\mathbf y)- \mathbf {y} )  \nu (\mathbf y) \, d\mathbf y = \int_\Lambda \!  \xi (\mathbf y) \cdot ( \mathbf g_s(\mathbf y) - \mathbf{y} )  \nu (\mathbf y) \, d\mathbf y = s\int_\Lambda \! \xi (\mathbf y) \cdot \nabla \varphi (\mathbf y)  \nu (\mathbf y) \, d\mathbf y .\]
Combining this with (\ref{newinequality}), we therefore obtain
 \begin{align*}
 -s\int_\Lambda \!  \xi (\mathbf y) \cdot \nabla \varphi (\mathbf y)  \nu (\mathbf y) \, d\mathbf y & + s^2\int_\Lambda \! |\nabla \varphi (\mathbf y)|^2\nu (\mathbf y) \, d\mathbf y \leqslant H(\nu ) - H(\nu _s)\\
 &\leqslant  {\cal E}(\nu ,\optfreem_s) - {\cal E}(\nu _s, \optfreem_s) \\
 &= \int_\Lambda \! \tilde{c}(\mathbf y, \mathbf S _\nu ^{\optfreem_s}(\mathbf y)) \nu (\mathbf y)  \, d\mathbf y - \int_\Lambda \! \tilde{c}(\mathbf y, \mathbf S _{\nu _s} ^{\optfreem_s} (\mathbf y)) \nu _s(\mathbf y)  \, d\mathbf y \\
& \leqslant \int_\Lambda \! \tilde{c}(\mathbf y, \mathbf S _{\nu _s}^{\optfreem _s}\circ \mathbf g_s(\mathbf y)) \nu (\mathbf y)  \, d\mathbf y - \int_\Lambda \! \tilde{c}(\mathbf y, \mathbf S _{\nu _s} ^{\optfreem _s} (\mathbf y)) \nu _s(\mathbf y)  \, d\mathbf y 
\end{align*}
\begin{equation}
 = \int_\Lambda \! \tilde{c}(\mathbf g_s^{-1}(\mathbf y), \mathbf S _{\nu _s}^{\optfreem_s}(\mathbf y)) \nu _s(\mathbf y)  \, d\mathbf y - \int_\Lambda \! \tilde{c}(\mathbf y, \mathbf S _{\nu _s} ^{\optfreem_s} (\mathbf y)) \nu _s(\mathbf y)  \, d\mathbf y ,
\label{newfloaty}
\end{equation}
since $\mathbf g_s \textrm{\#} \nu = \nu _s$.  Here $\mathbf S _{\nu }^{\optfreem_s}$ denotes the optimal transport map from $\nu $ to $\optfreem_s$ and $\mathbf S _{\nu _s}^{\optfreem_s}$  denotes the optimal transport map from $\nu _s$ to $\optfreem_s$ with respect to the cost function $\tilde{c}(\cdot , \cdot )$.  The existence of $\mathbf S _{\nu }^{\optfreem_s}$ and $\mathbf S _{\nu _s}^{\optfreem_s}$ follows from Theorem \ref{new3.3}.

Note that 
 \[ \mathbf g_s^{-1}(\mathbf y) = \mathbf y - s\nabla \varphi (\mathbf y) + \frac{s^2}{2}\nabla ^2\varphi (\mathbf y)\nabla \varphi (\mathbf y) + \epsilon (s, \mathbf y),\]
where $\epsilon $ is a function such that $|\epsilon (s, \mathbf y) | \leqslant |s|^3\| \varphi \| _{C^3(\mathbb{R}^3)}$.  

Combining this expression for $\mathbf g_s^{-1}$ with (\ref{newfloaty}) and using $\frac{\partial }{\partial y_3}\varphi = 0$, we conclude that
\begin{align*}
-s\int_\Lambda \!&  \xi (\mathbf y) \cdot \nabla \varphi (\mathbf y)  \nu (\mathbf y) \, d\mathbf y  + s^2\int_\Lambda \! |\nabla \varphi (\mathbf y)|^2\nu (\mathbf y) \, d\mathbf y \\ &\leqslant \int_\Lambda \! \left [ \tilde{c}(\mathbf g_s^{-1}(\mathbf y), \mathbf S _{\nu _s}^{\optfreem_s}(\mathbf y)) - \tilde{c}(\mathbf y, \mathbf S _{\nu _s} ^{\optfreem_s} (\mathbf y)) \right ]\nu _s(\mathbf y)  \, d\mathbf y\\
& = \int_\Lambda \! \left [\frac{1}{2} \left \{ \left | (g_s)_ 1^{-1}(\mathbf y) - (S_{\nu _s} ^{\optfreem_s})_1 (\mathbf y)\right | ^2 + \left | (g_s)_2^{-1}(\mathbf y) - (S_{\nu _s} ^{\optfreem_s})_2 (\mathbf y)\right | ^2 \right \} - y_3(S_{\nu _s} ^{\optfreem_s})_3(\mathbf y) \right. \\ &\qquad - \left. \frac{1}{2} \left \{ \left | y_1 - (S_{\nu _s} ^{\optfreem_s})_1 (\mathbf y)\right | ^2 + \left | y_2 - (S_{\nu _s} ^{\optfreem_s})_2 (\mathbf y)\right | ^2 \right \} - y_3(S_{\nu _s} ^{\optfreem_s})_3(\mathbf y) \right ] \nu _s(\mathbf y)\, d\mathbf y\\
& = \int_\Lambda \! \left [\frac{1}{2}\left \{ \left | y_1 - s\frac{\partial }{\partial y_1}\varphi (\mathbf y) - (S_{\nu _s}^{\optfreem_s})_1 (\mathbf y)\right | ^2 + \left | y_2 - s\frac{\partial }{\partial y_2}\varphi (\mathbf y)- (S_{\nu _s}^{\optfreem_s})_2 (\mathbf y)\right | ^2\right \} \right. \\ & \qquad - \left. \frac{1}{2}\left \{ \left |y_1 - (S_{\nu _s} ^{\optfreem_s})_1 (\mathbf y) \right | ^2 + \left | y_2 - (S_{\nu _s} ^{\optfreem_s})_2 (\mathbf y)\right | ^2\right \} \right ] \nu _s(\mathbf y)\, d\mathbf y + o(s)\\
&= s  \int_\Lambda \! \left (\mathbf S_{\nu _s} ^{\optfreem_s}(\mathbf y) - \mathbf y \right ) \cdot \nabla \varphi (\mathbf y) \, \nu _s (\mathbf y)\, d\mathbf y + o(s).
\end{align*}
By the definitions of $\mathbf g_s$ and $\nu _s$, we have that $\nu _s \rightarrow \nu $ in $\mathcal{P}_{ac}(\Lambda )$ as $s\to \infty $.  Then, by Lemma \ref{new4.3}, we have that $\sigma_{s} \rightarrow \sigma_{\bar h} $ in $\Freem $ as $s \rightarrow 0$, where $\sigma_{\bar h} $ denotes the unique minimiser in (\ref{newstablefirst}).  
Hence, dividing both sides first by $s>0$, then by $s<0$ and letting $|s| \rightarrow 0$, we use the natural stability of optimal maps to obtain
\[ -\int_\Lambda \!  \xi (\mathbf y) \cdot \nabla \varphi (\mathbf y) \,  \nu (\mathbf y) \,  d\mathbf y= \rho _0 \int_\Lambda \! \left (  \mathbf S_{\nu } ^{\optfreemin  }(\mathbf y) - \mathbf y \right ) \cdot \nabla \varphi (\mathbf y)  \, \nu  (\mathbf y)\, d\mathbf y.\]
Thus, we have that $\tilde{J}(\pi _\nu \xi (\mathbf y)) = \tilde{J}\left (\rho _0\left (  \mathbf y - \mathbf S_{\nu } ^{\optfreemin  }(\mathbf y) \right )\right )$, where $\pi _\nu : L^2(\nu ; \Lambda ) \rightarrow T_\nu \mathcal{P}_{ac}^2(\Lambda )$ denotes the canonical orthogonal projection
, with the tangent space defined by (\ref{tangent}).  The minimality of the norm of $\partial _0 H$ then gives
\begin{equation}\label{newoursupdiff}
\tilde{J}(\partial _0 H(\nu )) = \tilde{J}\left (\rho _0\left (   \mathbf y - \mathbf S_{\nu } ^{\optfreemin  }(\mathbf y) \right ) \right ) = \mathbf w(\mathbf y),
\end{equation}
where $\mathbf w$ is defined as in (\ref{newdualinc2}).

We can now check directly that conditions (H1) and (H2) hold.  
  Condition (H1) follows from the Theorem \ref{new3.3}, which tells us that the optimal  map $\mathbf S_{\nu } ^{\optfreemin  }$ is the gradient of a convex function.  
   Condition (H2) follows from the 
 stability of optimal maps (see \cite{ambbook}).  
Hence we may apply 
 the result of Theorem \ref{ag6.6} to conclude that there exists a Hamiltonian flow $\nu _{(t)}$ such that
\[ \frac{d}{dt}\nu _{(t)} + \nabla \cdot (\tilde{J}( \mathbf v_{(t)}) \nu _{(t)}) = 0, \qquad \nu _{(0)} = \nu _0, \qquad t \in (0, \tau )\]
where $\tilde{J} (\mathbf v_{(t)}) = \tilde{J}( \partial _0 H(\nu _{(t)}) )$ for a.e. $t \in [0, \tau ]$.  By (\ref{newoursupdiff}), this then completes the proof that the dual space continuity equation (\ref{newdualinc1}), with velocity field defined as in (\ref{newdualinc2}), is satisfied.  In addition, from (H3) and the definition of $\tilde{J}$, the energy associated with the flow is conserved.

The boundedness of the support of $\nu_{(t)}$ also follows Theorem (\ref{ag6.6}).
\end{proof}

\subsubsection*{Proof of the main Theorem \ref{new5.5}}
From the definition of $\mathbf w$ in Proposition \ref{newmainalt}, we have that $(\freemin  , \mathbf T)$ is a stable solution of (\ref{newdualinc1})-(\ref{newdualinc6}), where $\mathbf T = \nabla P$ (see Theorem \ref{new3.3}).  Theorem \ref{new5.5} (\emph{i}) follows from (\ref{bound1}), (\ref{bound2}); Theorem \ref{new5.5} (\emph{ii}) follows from Theorem \ref{new4.2}; Theorem \ref{new5.5}. Note also that by the definition of $\mathbf w$ in terms of the optimal map $\mathbf T^{-1}$ which, by Theorem \ref{new3.3}, is the gradient of a convex function. 
\qed

\bigskip
In summary,
by rewriting the incompressible semi-geostrophic equations in an appropriate set of geostrophic coordinates and reformulating the problem as a coupled optimal transport/continuity problem, we have been able to show the existence of stable weak solutions in dual space. 

\section{The free boundary problem for the compressible semi-geostrophic system}
\setcounter{equation}{0}
In this section, we generalise the proof of the previous section to hold for the compressible system (\ref{commom})-(\ref{comstate}). The boundary conditions are (\ref{cgspeedboundary})-(\ref{cgboundcon}), as before.

The results on the existence of  dual solutions for the compressible free boundary problem  rely on formulating the equations in the so-called pressure coordinates.  In this form, the problem in dual coordinates is formally identical to the one of the previous section for the incompressible case,  but formulated with respect to a different  cost.

We will formulate the problem, and state the main result. All details can be found in \cite{dkgthesis}.

\subsection{Formulation in pressure coordinates}
We consider the fully compressible system (\ref{commom})-(\ref{comstate}). 
The equations are to be solved in the variable domain defined by (\ref{cgfreeOmega}).

The \emph{geostrophic energy} associated with the flow is defined as
\begin{equation}\label{geoenergy}
E(t)=\int_{\Omega } \! \left [ \frac{1}{2}\left |\mathbf u^g \right |^2(t,\mathbf x) + \phi(\mathbf x) + c_v\theta(t,\mathbf x) \left (\frac{p(t,\mathbf x)}{p_\textrm{\scriptsize{ref}}}\right )^{\frac{\kappa -1}{\kappa }} \right ] \rho(t,\mathbf x) \, d \mathbf x.
\end{equation}

\smallskip
It follows from the hydrostatic balance approximation $\frac{\partial p}{\partial x_3}=-\rho$  that $\frac{\partial p}{\partial x_3}$   is always negative. 
 Hence, the change of variables $x_3 = x_3(p)$ is well-defined and we can express any function $\psi $ of $(t, x_1, x_2, x_3)$ in terms of $(t, x_1, x_2, p)$ by considering $x_3$ as a dependent variable.

 In these new coordinates, the compressible semi-geostrophic equations take the form
\be
\left\{\begin{array}{ll}
 \frac{D\p u_1^g}{Dt} -  u_2 +  \frac{\partial \phi }{\partial x_1} = 0,&\\
\frac{D\p u_2^g}{Dt} +  u_1 + \frac{\partial \phi }{\partial x_2} = 0,&\\
\nabla \p \cdot \mathbf u_p = 0,
&\quad (t, \mathbf x\p)\in\;[0,\tau)\times \Omega \p (t),\\
 \frac{D\p \theta }{Dt} = 0,&\\
 u_1^g = -\frac{\partial \phi }{\partial x_2}, \qquad  u_2^g = \frac{\partial \phi }{\partial x_1},&\\
\frac{\partial \phi}{\partial p} = -\frac{R\theta p^{\kappa -1}}{\pr^\kappa },&
\end{array}\right.
\ee
where $\mathbf u \p = (u_1, u_2, \omega )$ denotes the velocity in pressure coordinates,
\[ \Omega \p (t) = \{ (x_1, x_2, p) \in \mathbb{R}^3 : (x_1, x_2) \in \Omega _2, p_h \leqslant p \leqslant p\s (t, x_1, x_2) \} ,\]
and $p\s $ is unknown,  while $p_h\geq 0$ is the constant pressure at the fixed boundary. 
 
The boundary conditions read
\begin{eqnarray}\label{freecombound1}
&&\mathbf u\p \cdot \mathbf n = 0 \qquad (x_1,x_2,p)\in \partial \Omega\p (t)\setminus \{p=p\s\},
\\
\label{freecombound}
&&x_3 = 0, \qquad \frac{D\p p\s }{Dt} = \omega,\quad {\rm for}\;p=p\s.
\end{eqnarray}

 We are also given the initial condition
\begin{equation}\label{psinitial}
\ps (0, \cdot ) = (\ps )_0(\cdot ) \in  C(\Omega _2 ) \cap W^{1, \infty }(\Omega _2 ).
\end{equation}

The energy associated with the flow, in $\mathbf x\p $ coordinates, 
takes the form
\begin{align}
\notag E \p &= \int_{\Omega \p(t)} \! \left[ \frac{1}{2 }((u_1^g)^2 + (u_2^g)^2) + \frac{c_p\theta p^\kappa }{ \pr^\kappa }\right] \, d\mathbf x\p \\
&\label{freecomenergy}= \int_{\Omega _2} \int_{p_h}^{p\s (t, x_1, x_2)}\! \left[ \frac{1}{2 }((u_1^g)^2 + (u_2^g)^2) + \frac{c_p\theta p^\kappa }{ \pr^\kappa }\right] \, dx_1dx_2dp.
\end{align}
 
 \subsubsection{Formulation in dual coordinates}
Assume that $\Lambda \subset \mathbb{R}^3$ is an open bounded set.  As in Section \ref{secnewgeo}, we perform a change of variables to geostrophic coordinates $\mathbf y \in \Lambda $.  This change of variables is now given by
\begin{equation}\label{freecomchange}
y_1 = x_1 + u_2^g, \qquad y_2 = x_2 - u_1^g, \qquad y_3 = -\frac{c_p\theta }{  \pr ^\kappa }. 
\end{equation}
We will denote by $\mathbf T$ the change of variables from physical to geostrophic coordinates, i.e.
\[ \mathbf T(t, \mathbf x\p ) = (T_1(t, \mathbf x\p ), T_2(t, \mathbf x\p ), T_3(t, \mathbf x\p )) = (y_1, y_2, y_3). \]

We use (\ref{freecomchange}) to rewrite the energy in (\ref{freecomenergy}) as
\begin{eqnarray}
 E\p &&= \int_{\Omega _2}\int_{p_h}^{\pfreemin } \! \left[ \frac{1}{2}\{|x_1 - y_1|^2 + |x_2 - y_2|^2 \} - p^\kappa y_3 \right] \, d\mathbf x\p \nonumber \\
&&= \int_{\mathbb {R}^3} \! \left[ \frac{1}{2}\{|x_1 - T_1(\mathbf x\p )|^2 + |x_2 - T_2(\mathbf x \p )|^2 \} - p^\kappa T_3(\mathbf x\p )\right] \optfreemin \, d\mathbf x\p \label{freecomenergyrewrite}.
\end{eqnarray}
where $\optfreemin := \chi _{\Omega _2 \times [p_h, \pfreemin ]}$.

Define
	\begin{equation}\label{freecomfreemset}
H_p:= \left \{  \pfreem : [0, \tau ) \times \Omega _2 \to (0, \infty ),  \pfreem \in W^{1,\infty}(\Omega _2),  \|\pfreem\|_1=1,\;\int_{\R^3}x_3d\sigma_{p_s}\leq M_0\right \},
\end{equation} 
and
\be
 \pFreem :=\{ \optfreem_{\pfreem}(t,\cdot)  \in \optFreem \;\bigg| p_s\in \,H_p\}.
 \label{calHp}\ee
 This space is the analogue of (\ref{calH}), but with respect to pressure coordinates. As in Section \ref{Hsect}, it can be shown that the space $\pFreem$ is compact in $\optFreem$.
One can again appeal to Lemma \ref{lemmasingleval} to show that, in order for $\ps $ to correspond to an energy minimiser, $\ps $ must be a well-defined single valued function. 

Define the potential density $\nu := \mathbf T \textrm{\#}\optfreemin $ as the push forward of the measure $\optfreemin $ under the map $\mathbf T$.  

Then, given $\nu \in \mathcal{P}_{ac}(\Lambda )$, we define for any $\pfreem \in \pFreem $, the functional
\begin{equation}\label{freecomenergymin}
E_\nu (\pfreem ) = \inf _{\overline{\mathbf T} \textrm{\#} \optfreem = \nu }\int_{\mathbb{R}^3} \! c(\mathbf x\p , \overline{\mathbf T}(\mathbf x\p ))\optfreem (\mathbf x\p )\, d\mathbf x\p = \inf _{\overline{\mathbf T} \textrm{\#} \optfreem = \nu }\int_{\Omega _2 }\int _{p_h}^{\pfreem } \! c(\mathbf x\p , \overline{\mathbf T}(\mathbf x\p )) \, d\mathbf x\p ,
\end{equation}
where
\begin{equation}\label{freecomcost}
c(\mathbf x\p , \mathbf y) = \left [\frac{1}{2}\{|x_1 - y_1|^2 + |x_2 - y_2|^2 \} - p^\kappa y_3 \right ] .
\end{equation}

The analogue of Principle \ref{princ1} (Cullen's stability principle) now holds for $E\p$.
Hence the dual space semi-geostrophic system then takes the form

%

\begin{eqnarray}
\nonumber
&& \frac{\partial \nu }{\partial t} + \nabla \cdot (\nu \mathbf w) = 0,\\
\nonumber
&& \mathbf w(t, \mathbf y)= \fc J(\mathbf  y - \mathbf T^{-1}(t, \mathbf y)) ,\\
\label{freecomdual3}
\nonumber&&\mathbf T(t, \cdot ) \textrm{ is the unique optimal map from $\chi _{\Omega _2 \times [p_h, \pfreemin ]}$ to $\nu $ with cost (\ref{freecomcost})},\\
\nonumber
&& \pfreemin (t, \cdot ) \textrm{ minimises } E_{\nu(t, \cdot )} (\cdot ) \textrm{ over } \pFreem ,\\
\end{eqnarray}
where
\be\nu (0, \cdot ) = \nu _0 (\cdot ) \textrm{ compactly supported probability density in } L^r, \, r \in (1, \infty ).\ee

\subsubsection{The main existence theorem}

\begin{theorem}\label{freecom5.5}

Let $1 \leq  r < \infty $ and let $\nu _0 \in L^r (\Lambda _0 )$ be an initial potential density with support in $\Lambda _0$, where $\Lambda _0 $ is a bounded open set in $\mathbb{R} ^3$. 
Let $c(\cdot ,\cdot )$ be given by (\ref{freecomcost}).  Then the system of semi-geostrophic equations in dual variables (\ref{freecomdual3}) has a stable weak solution $(\pfreemin , \mathbf T)$ with $\pfreemin(t,\cdot)\in \;W^{1,\infty}(B(0,S))$ for some $S>0$.  

For $\nu  = \mathbf T \textrm{\#} \optfreemin $, where $\optfreemin = \chi _{\Omega _2 \times [p_h, \pfreemin ]}$, and $\mathbf w$ as in (\ref{newdualinc2}), this solution satisfies
\begin{enumerate}[label=(\textit{\roman{*}}), ref=\textit{(\roman{*})}]
	\item \[ \nu (\cdot , \cdot ) \in L^r ((0, \tau ) \times \Lambda ), \qquad \left\| \nu (t, \cdot )\right\| _{L^r (\Lambda )} \leqslant \left\| \nu _0 (\cdot )\right\| _{L^r (\Lambda )}, \qquad \forall\; t \in [0, \tau ],
\]
\item  \[ \phi (t, \cdot ) \in W^{1, \infty } (\Omega _2 \times [p_h, \pfreemin ]), \quad \left\| \phi (t, \cdot )\right\| _{W^{1, \infty } (\Omega _2 \times [p_h, \pfreemin ])} \leqslant C = C(\Omega _2 \times [p_h, \pfreemin ], \Lambda, c(\cdot , \cdot )),\] \[ \forall\;t \in [0, \tau ], \]
\item \[ \left\| \mathbf w(t, \cdot ) \right\| _{L^\infty (\Lambda )} \leqslant C = C(\Omega _2 \times [p_h, \pfreemin ], \Lambda ),\qquad \forall\;t \in [0, \tau ],  \]
\end{enumerate}
where $\Lambda $ is a bounded open domain in $\mathbb{R}^3$ containing $supp (\nu )$.
\end{theorem}

\section*{Conclusions}

We have given a rigorous proof of the existence of dual space solutions for the semigeostrophic system posed in a domain with variable height, in three dimension. 
The proof builds on the previous techniques introduced by Benamou-Brenier and Cullen, Gangbo and Maroofi, and it makes use of the general theory of Hamiltonian flows of Ambrosio-Gangbo.  
 The  proof is given in detail for the incompressible case, where we outline all the difficulties that need to be overcome. A similar proof then holds also for the compressible set of equations, whose formal structure in pressure coordinates is analogous to the structure of the incompressible ones. 

We expect that it should be possible to use arguments similar to the ones in \cite{feldman} to extend the validity of the main result presented here to the case of weak Lagrangian solutions in physical space.

\section{Acknowledgements}

DKG gratefully acknowledges the support of an EPSRC-CASE studentship sponsored by the MET Office.

\bibliographystyle{acm}
\bibliography{semigeostrophicbib}

\begin{thebibliography}{10}

\bibitem{ambrosio}
{\sc Ambrosio, L.}
\newblock Transport equation and cauchy problem for {B}{V} vector fields.
\newblock {\em Invent. Math. 158\/} (2004), 227--260.

\bibitem{ambperimetri}
{\sc Ambrosio, L.}
\newblock La teoria dei perimetri di caccioppoli--de giorgi ei suoi pi{\`u}
  recenti sviluppi.
\newblock {\em Rendiconti Lincei-Matematica e Applicazioni 21}, 3 (2010),
  275--286.

\bibitem{ambnew}
{\sc Ambrosio, L., Colombo, M., De~Philippis, G., and Figalli, A.}
\newblock Existence of eulerian solutions to the semigeostrophic equations in
  physical space: the 2-dimensional periodic case.
\newblock {\em Communications in Partial Differential Equations 37}, 12 (2012),
  2209--2227.

\bibitem{ambnew2}
{\sc Ambrosio, L., Colombo, M., De~Philippis, G., and Figalli, A.}
\newblock A global existence result for the semigeostrophic equations in three
  dimensional convex domains.
\newblock {\em arXiv preprint arXiv:1205.5435\/} (2012).

\bibitem{ambgangbo}
{\sc Ambrosio, L., and Gangbo, W.}
\newblock Hamiltonian {O}{D}{E}'s in the wasserstein space of probability
  measures.
\newblock {\em Comm. Pure. Appl. Math. 61\/} (2008), 18--53.

\bibitem{ambbook}
{\sc Ambrosio, L., Gigli, N., and Savar\'{e}, G.}
\newblock {\em Gradient {F}lows in {M}etric {S}paces and in the {S}pace of
  {P}robability {M}easures}, vol.~58 of {\em Lectures in Mathematics ETH
  Z\"{u}rich}.
\newblock Birkh\"{a}user Verlag, Basel, 2005.

\bibitem{benamou}
{\sc Benamou, J.-D., and Brenier, Y.}
\newblock Weak existence for the semigeostrophic equations formulated as a
  coupled {M}onge-{A}mp\`{e}re/transport problem.
\newblock {\em SIAM J. Appl. Math. 58\/} (1998), 1450--1461.

\bibitem{brenier}
{\sc Brenier, Y.}
\newblock Polar factorization and monotone rearrangement of vector-valued
  functions.
\newblock {\em Comm. Pure Appl. Math. 44\/} (1991), 375--417.

\bibitem{caffarelli2010free}
{\sc Caffarelli, L.~A., and McCann, R.~J.}
\newblock Free boundaries in optimal transport and {M}onge-{A}mp\`ere obstacle
  problems.
\newblock {\em Ann. of Math.(2) 171}, 2 (2010), 673--730.

\bibitem{sedjro}
{\sc Cullen, M., and Sedjro, M.}
\newblock On a model of forced axisymmetric flows.
\newblock {\em SIAM Journal on Mathematical Analysis 46}, 6 (2014), 3983--4013.

\bibitem{cullen}
{\sc Cullen, M. J.~P.}
\newblock {\em A mathematical theory of large-scale atmosphere/ocean flow}.
\newblock Imperial College Press, 2006.

\bibitem{feldman}
{\sc Cullen, M. J.~P., and Feldman, M.}
\newblock Lagrangian solutions of semigeostrophic equations in physical space.
\newblock {\em SIAM J. Math. Anal. 37\/} (2006), 1371--1395.

\bibitem{gangbo}
{\sc Cullen, M. J.~P., and Gangbo, W.}
\newblock A variational approach for the 2-dimensional semi-geostrophic shallow
  water equations.
\newblock {\em Arch. Rational Mech. Anal. 156\/} (2001), 241--273.

\bibitem{CGP1}
{\sc Cullen, M. J.~P., Gilbert, D.~K., and Pelloni, B.}
\newblock Solution of the fully compressible semi-geostrophic system.
\newblock {\em Comm. {P}{D}{E} 39\/} (2014), 591--625.

\bibitem{maroofi}
{\sc Cullen, M. J.~P., and Maroofi, H.}
\newblock The fully compressible semi-geostrophic system from meteorology.
\newblock {\em Arch. Rational Mech. Anal. 167\/} (2003), 309--336.

\bibitem{purser}
{\sc Cullen, M. J.~P., and Purser, R.~J.}
\newblock An extended {L}agrangian theory of semi-geostrophic frontogenesis.
\newblock {\em J. Atmos. Sci. 41\/} (1984), 1477--1497.

\bibitem{eliassen}
{\sc Eliassen, A.}
\newblock The quasi-static equations of motion.
\newblock {\em Geofys. Publ. 17\/} (1948).

\bibitem{faria2}
{\sc Faria, J. C.~O.}
\newblock On the existence and weak stability of solutions to the cmpressible
  semigeostrophic equations.
\newblock {\em J Math. Anala Appl. 406\/} (2013), 447--463.

\bibitem{faria}
{\sc Faria, J. C.~O., Lopes~Filho, M.~C., and Nussenzveig~Lopes, H.~J.}
\newblock Weak stability of {L}agrangian solutions to the semigeostrophic
  equations.
\newblock {\em Nonlinearity 22\/} (2009), 2521--2539.

\bibitem{dkgthesis}
{\sc Gilbert, D.~K.}
\newblock {\em Analysis of large-scale atmospheric flows}.
\newblock PhD thesis, University of Reading, 2013.

\bibitem{hoskins}
{\sc Hoskins, B.~J.}
\newblock The geostrophic momentum approximation and the semi-geostrophic
  equations.
\newblock {\em J. Atmos. Sci. 32\/} (1975), 233--242.

\bibitem{loeper}
{\sc Loeper, G.}
\newblock A fully non-linear version of the incompressible {E}uler equations:
  the semi-geostrophic system.
\newblock {\em SIAM J. Math. Anal. 38\/} (2006), 795--823.

\bibitem{lopes}
{\sc Lopes~Filho, M., and Lopes, H.~N.}
\newblock Existence of a weak solution for the semigeostrophic equation with
  integrable initial data.
\newblock In {\em Proceedings of the Royal Society of
  Edinburgh-A-Mathematics\/} (2002), vol.~132, Cambridge Univ Press,
  pp.~329--340.

\bibitem{shutts}
{\sc Shutts, G.~J., and Cullen, M. J.~P.}
\newblock Parcel stability and its relation to semigeostrophic theory.
\newblock {\em J. Atmos. Sci. 44\/} (1987), 1318--1330.

\bibitem{villanioldnew}
{\sc Villani, C.}
\newblock {\em Optimal transport, old and new}.
\newblock Springer Verlag, 2008.

\end{thebibliography}

\newpage

\section*{Appendix}
\setcounter{equation}{0}

\subsection*{Useful Conventions, Notation and Definitions}
We list here notation and conventions used in the paper. 

 \begin{equation}\label{list}
{\bf Physical\;variables \;and \;constants}\end{equation}

\begin{enumerate}[label=(\textit{\roman{*}}), ref=\textit{(\roman{*})}]
	\item $\Omega $ denotes an open bounded convex set in $\mathbb{R} ^3 $, representing the physical domain containing the fluid; $\tau>0$ is a fixed positive constant; all functions in physical coordinates are defined for $(t, \mathbf x )\in [0,\tau)\times \Omega $;
	\item $\mathbf u(t, \mathbf x)=(u_1(t, \mathbf x),u_2(t, \mathbf x),u_3(t, \mathbf x))$ represents the full velocity of the fluid;
	\item $\mathbf u^g(t, \mathbf x) = (u_1^g(t, \mathbf x), u_2^g(t, \mathbf x), 0)$ represents the (two-dimensional) geostrophic velocity;
	\item $p(t, \mathbf x)$ represents the pressure;
	\item $\rho (t, \mathbf x)$ represents the density;
	\item $\theta (t, \mathbf x)$ represents the potential temperature. Given its physical meaning, we assume $\theta (t, \mathbf x)$  to be strictly positive and bounded;
		\item $\phi $ is the prescribed geopotential.  We assume that $\phi = \gg x_3$, where $\gg $ denotes the constant acceleration due to gravity;
			\item $f_\textrm{\scriptsize{cor}}$ denotes the Coriolis parameter, which we assume to be constant; in all that follows, we assume $\fc= 1$;
		\item $p_\textrm{\scriptsize{ref}}$ is the reference value of the pressure; $R$ represents the gas constant. 
\end{enumerate}

 \begin{equation}\label{list2}
{\bf Notations \; and \; other \; conventions}\end{equation}
 
\begin{itemize}
	
	\item The Lebesgue measure of any set $A$ in $\mathbb{R}^3$ will be denoted by $|A|$.

\item[(a)]  Given an open set $A$ in $\mathbb{R}^3$, we will denote by
	 
	 - \quad $\chi _A$ - the  characteristic function of $A$;
	 
	  - \quad $P_{ac}(A)$  - the set of probability measures in $\mathbb{R}^3$ with supports contained in $A$, absolutely continuous with respect to Lebesgue measure.
\item Given some function $H:A \rightarrow (-\infty , +\infty ]$, we denote by $D(H)$ the set of all $a \in A$ such that $H(a)< +\infty $.  We say that $H$ is proper if $D(H) \neq \emptyset $.
	\item[(b)] Unless otherwise specified, measurable means Lebesgue measurable and $a.e.$ means Lebesgue-$a.e.$
	\item[(c)] $D_t$ denotes the Lagrangian derivative, defined as $D_t = \partial _t + \mathbf u \cdot \nabla $, where ${\bf u}$ denotes the full velocity of the flow as in {\em (ii)}.  
	\item[(d)] For convenience, we will sometimes use the notation $F_{(t)}(\cdot ) = F(t, \cdot )$ to denote the map $F$ evaluated at fixed time $t$.  
	\item[(e)]
	$W^{1,\infty}$ denotes the usual Sobolev space of essentially bounded functions with first weak derivative in $L^\infty$.
	\end{itemize}

%


\subsection*{Optimal transport}\label{opttr}
For all general definitions regarding probability measures, classical existence of the solution of the optimal transport problem with respect to a quadratic cost,  and the Wasserstein metric, we refer to  \cite{ambbook}. In this section we only discuss results we need to generalise for the purpose of  the present paper.

\medskip

The optimal transport results of \cite{CGP1,maroofi}, that provide the basis for the present results, utilise the Kantorovich dual problem and many useful properties of its $c$-transform solutions.  We start by defining the $c$-transforms of  functions $f,g:\R^3\to\R$
\begin{equation}\label{new29}
f^c(\mathbf y) := \inf _{\mathbf x \in \mathbb{R}^3} \{c(\mathbf x, \mathbf y) - f(\mathbf x)\}
\end{equation} and 
\begin{equation}\label{new30}
g^c(\mathbf x) := \inf _{\mathbf y \in\R^3} \{c(\mathbf x, \mathbf y) - g(\mathbf y)\} ,
\end{equation}
for some cost function $c(\cdot , \cdot )$.  We say that $f$ is $c-$concave if and only if $f$ = $g^c$ for some function $g$.

For  the cost function we will consider here, given by (\ref{newcost}), we have the following useful characterisation of $c-$transforms which allows us to conclude that the optimal map in Theorem \ref{new3.3} is indeed the gradient of a convex function.

\begin{lemma}\label{newcconcavechar}
Let $\Lambda $ be a bounded open set in $\mathbb{R}^3$ and let $c(\mathbf x, \mathbf y)$ be given by (\ref{newcost}).  Then $f$ is a $c-$concave function from $\mathbb{R}^3$ into $\mathbb{R}$ if and only if $\mathbf x \mapsto \overline{P}(\mathbf x)$, defined by
\begin{equation}\label{generalP}
\overline{P}(\mathbf x) := -f(\mathbf x) + \frac{1}{2}(x_1^2 + x_2^2)
\end{equation}
is convex. 

\end{lemma}
\begin{proof}
We know that $f$ is $c-$concave if and only if $f = g^c$ for some function $g$ defined on a bounded set $\overline{\Lambda}\subset \R^3$ into $\mathbb{R}$, i.e.
\begin{eqnarray*}
f(\mathbf x) &&= \inf _{\mathbf y \in \overline{\Lambda}} \{ c(\mathbf x, \mathbf y) - g(\mathbf y) \} \\
   &&= \inf _{\mathbf y \in \overline{\Lambda}} \left\{\left [\frac{1}{2}\{|x_1 - y_1|^2 + |x_2 - y_2|^2 \} - x_3y_3 \right ] - g(\mathbf y) \right\} \\
  &&= \inf _{\mathbf y \in \overline{\Lambda}} \left\{ \left [\frac{1}{2}(x_1^2 + x_2^2) + \frac{1}{2}(y_1^2 + y_2^2) - \mathbf x \cdot \mathbf y \right ] - g(\mathbf y) \right\} ,
   \end{eqnarray*}
which holds if and only if
\[ f(\mathbf x) - \frac{1}{2}(x_1^2 + x_2^2) = \inf _{\mathbf y \in \overline{\Lambda}}  \left\{-\mathbf x \cdot \mathbf y - g_0(\mathbf y) + \frac{1}{2}(y_1^2 + y_2^2) \right\} , \]
i.e.
\[ \frac{1}{2}(x_1^2 + x_2^2) - f(\mathbf x)  = \sup _{\mathbf y \in \overline{\Lambda}} \bigg\{ \mathbf x \cdot \mathbf y - \left( \frac{1}{2}(y_1^2 + y_2^2) -g(\mathbf y)\right) \bigg\} .\]
Defining
\begin{equation}\label{generalR}
\overline{R}(\mathbf y) := -g(\mathbf y) + \frac{1}{2}(y_1^2 + y_2^2)
\end{equation}
we see that $f$ is $c-$concave if and only if $\overline{P}$ is the Legendre transform of some function $\overline{R}$, i.e. if and only if $\overline{P}$ is convex. 
\end{proof}

\subsection*{Hamiltonian Flows}\label{existag}
The semi-geostrophic problem can be formulated as coupling an energy minimisation problem with  a transport equation, with certain specific regularity properties.  The original proof in \cite{benamou} used a time-discretisation argument to prove the solution of the relevant  transport equation exists. However, using a more recent result of Ambrosio and Gangbo \cite{ambgangbo} of Hamiltonian ODEs in the Wasserstein space of probability measures, it can be shown that  the solution of the energy minimisation problem yields a solution of the associated transport equation, through the fact that the velocity field is precisely realised as the superdifferential of the energy.  

Here, we summarise the results of \cite{ambgangbo} that we use in our main proof . While these results may appear  technical, they essentially state that  if the Hamiltonian of the system, i.e. the energy in dual space,  satisfies certain conditions,  then the Hamiltonian flow whose velocity field is given by the superdifferential of the energy exists. 

\medskip
In what follows, we let $\mu $, $\nu $, $\sigma $ be arbitrary measures belonging to $\mathcal{P}^2(\mathbb{R}^d)$, the space of probability measures on $\mathbb{R}^d$ with finite second order moments.
We define the tangent space to $\mathcal{P}^2(\mathbb{R}^d)$ at $\nu $ as
\begin{equation}\label{tangent}
T_{\nu _{(t)}}\mathcal{P}^2(\mathbb{R}^d) = \overline{ \{ \nabla \varphi : \varphi \in C_c^\infty (\mathbb{R}^d) \} }^{L^2(\nu ; \mathbb{R}^d)}.
\end{equation}
\begin{definition}\label{prope}
Given some function $H:A \rightarrow (-\infty , \infty ]$, we denote by $D(H)$ the set of all $a \in A$ such that $H(a)< \infty $.  We say that $H$ is proper if $D(H) \neq \emptyset $.
\end{definition}

The space $\mathcal{P}^2(\mathbb{R}^d)$, equipped with the Wasserstein metric $W_2$ 
is a complete and separable space, but is not locally compact since narrow convergence of measures does not necessarily imply convergence of second order moments.  Following \cite{ambrosio}, we generalise the notions of differentiability and convexity to the metric space $(\mathcal{P}^2(\mathbb{R}^d), W_2)$.  In what follows, we deal with concave rather than convex functions.  This is due to the way in which we define our Hamiltonian $H$ to represent the minimal energy associated with the flow.  Hence, in what follows we replace all definitions and results involving subdifferentiability and $\lambda -$convexity given in \cite{ambrosio} with results involving superdifferentiability and $\lambda -$concavity.  We also restrict our attention only to  measures which are absolutely continuous with respect to Lebesgue measure.

\begin{definition}\label{superdifferential}
Let $H: \mathcal{P}_{ac}^2(\mathbb{R}^3) \rightarrow (-\infty , \infty ]$ be a proper, upper semi-continuous function and let $\nu \in D(H)$.  We say that $\mathbf v \in L^2(\nu ; \mathbb{R}^3)$ belongs to the \textit{Fr\'{e}chet superdifferential} $\partial H(\nu )$ if
\begin{equation}H(\nut) \leqslant H(\nu ) + \int_{\mathbb{R}^3  } \!  \mathbf v(\mathbf y) \cdot (\mathbf R_{\nu }^{\nut}(\mathbf y) - \mathbf y )\, \nu (\mathbf y) \, d\mathbf y + o(W_2(\nu , \nut)) \quad \textrm{ as } \nut \rightarrow \nu ,\label{appendix}
\end{equation}
where $\mathbf R_\nu ^{\nut} $ is the optimal map in the transport of $\nu$  to  $\nut$.  We denote by $\partial _0 H(\nu )$ the element of $\partial H(\nu )$ of minimal $L^2(\nu ; \mathbb{R}^3)-$norm.  
\end{definition}

Note that, by the minimality of its norm, $\partial _0 H(\nu )$ belongs to $ \partial H(\nu ) \cap T_{\nu }\mathcal{P}_{ac}^2(\mathbb{R}^3)$.

\begin{definition}\label{lambdaconcave}
Let $H: \mathcal{P}_{ac}^2(\mathbb{R}^3) \rightarrow (-\infty , \infty ]$ be proper and let $\lambda \in \mathbb{R}$.  
We say that $H$ is \textit{$\lambda -$concave} if, for every $\tilde{\nu}_0$, $\tilde{\nu}_1 \in \mathcal{P}_{ac}^2(\mathbb{R}^3)$ denoting by $\bf T$ optimal map in the transport of $\tilde{\nu}_0$ to $\tilde{\nu}_1$, we have
\[ H(\nu _{(t)}) \geqslant (1 - t)H(\tilde{\nu}_0) + tH(\tilde{\nu}_1) - \frac{\lambda }{2}t(1 - t)W_2^2(\tilde{\nu}_0, \tilde{\nu}_1) \]
for all $t \in [0, 1]$, where $\nu _{(t)} ={ \bf T}\textrm{\#}\left[(1-t)\tilde\nu_0+t\tilde \nu_1\right].$
\end{definition}
\begin{proposition}\label{characterise}
Let $H: \mathcal{P}_{ac}^2(\mathbb{R}^3) \rightarrow (-\infty , \infty ]$ be upper semi-continuous and $\lambda -$concave for some $\lambda \in \mathbb{R}$ and let $\nu \in D(H)$.  Then, the following condition is equivalent to $\mathbf v \in \partial H(\nu )$: 
\[ H(\nut ) \leqslant H(\nu ) + \int_{\mathbb{R}^3  } \!  \mathbf v(\mathbf y) \cdot ( \mathbf R_{\nu }^{\nut}(\mathbf y) - \mathbf y)\nu (\mathbf y) \, d \mathbf y + \frac{\lambda }{2}W_2^2(\nu , \nut) \qquad \textrm{ for all } \nut \in \mathcal{P}_{ac}^2(\mathbb{R}^3),\]
where $\mathbf R_\nu ^{\nut}$ is the optimal map cin the transport of $\nu$  to  $\nut$. 
\end{proposition}

We have also the following useful result from \cite[Proposition 10.12]{villanioldnew} which provides us with a link between superdifferentiability and $\lambda -$concavity in the specific case when $\lambda = -2$ (i.e. semi-concavity):
\begin{proposition}\label{vill10.12}
Let $H: \mathcal{P}_{ac}^2(\mathbb{R}^3) \rightarrow (-\infty , \infty ]$ be a proper, upper semi-continuous function. If $H$ is locally superdifferentiable, then $H$ is also locally $(-2)-$concave.
\end{proposition}

We can now define Hamiltonian ODEs as follows:
\begin{definition}\label{hamode}
Let $H: \mathcal{P}_{ac}^2(\mathbb{R}^3) \rightarrow (-\infty , \infty ]$ be a proper, upper semi-continuous function.  Define the linear transformation $\tilde{J} : \mathbb{R}^3 \rightarrow \mathbb{R}^3$ by 
\begin{equation}\label{tildeJ}
\tilde{J}(v_1(\mathbf y), v_2(\mathbf y), v_3(\mathbf y)) = y_3(-v_2(\mathbf y), v_1(\mathbf y), 0),
\end{equation}
for all $\mathbf v(\mathbf y) \in \mathbb{R}^3$.  We say that an absolutely continuous curve $\nu _{(t)} : [0, \tau ] \rightarrow D(H)$ is a \textit{Hamiltonian ODE} relative to $H$, starting from $\nu _0 \in \mathcal{P}_{ac}^2(\mathbb{R}^3)$, if there exists $\mathbf v_{(t)} \in L^2(\nu _{(t)};\mathbb{R}^3)$ with $\left\| \mathbf v_{(t)} \right\| _{L^2(\nu _{(t)})} \in L^1(0, \tau )$, such that
\begin{equation}\label{33} \begin{cases}
\frac{d}{dt}\nu _{(t)} + \nabla \cdot (\tilde{J} \mathbf v_{(t)} \nu _{(t)}) = 0, \qquad \nu _{(0)} = \nu _0, \qquad t \in (0, \tau )\\
\mathbf v_{(t)} \in T_{\nu _{(t)}}\mathcal{P}_{ac}^2(\mathbb{R}^3) \cap \partial H(\nu _{(t)}) \qquad \textrm{for }a.e.\textrm{ } t.
\end{cases}\end{equation}
\end{definition}

 \smallskip
 We now consider Hamilton flows, as in the definition (\ref{hamode}) and the condition (H1), (H2), (H3) given in section \ref{newmain}.  The main result on these flows, which is used in our proof of the main theorem \ref{new5.5}, can be stated as follows:
\begin{theorem}\label{ag6.6}
Assume that (H1) and (H2) hold for $H(\nu )$ and that $\tau > 0$ satisfies
\begin{equation}\label{maxtime}
C_0 \tau \sqrt{24(1 + e^{(25C_0^2 + 1)\tau }(1 + M_2(\nu _0)))}<R_0.
\end{equation}
Then there exists a Hamiltonian flow $\nu _{(t)} \in \mathcal{P}_{ac}^2(\mathbb{R}^3 )$,  $\nu_{(t)}:[0, \tau] \rightarrow D(H)$
starting from 
$\nu_0 \in \mathcal{P}_{ac}^2(\mathbb{R}^3)$, satisfying (\ref{33}), such that the velocity field $\mathbf v_{(t)}$ coincides with $\partial _0 H(\nu _{(t)})$ for a.e. $t \in [0, \tau ]$.  Furthermore, the function $t \rightarrow \nu _{(t)}$ is Lipschitz continuous.
Finally, there exists a function $l(r)$ depending only on $\tau $ and $C_0$ such that
\begin{equation}\label{bound1} \nu _0 \geqslant m_r \, \, a.e. \textrm{ on }B_r \textrm{ for all }r>0 \quad \implies \quad \nu _{(t)} \geqslant m_{l(r)} \, \, a.e. \textrm{ on }B_r\textrm{ for all }r>0
\end{equation}
and
\begin{equation}\label{bound2} \nu _0 \leqslant M_r \, \, a.e. \textrm{ on }B_r \textrm{ for all }r>0 \quad \implies \quad \nu _{(t)} \leqslant M_{l(r)} \, \, a.e. \textrm{ on }B_r\textrm{ for all }r>0.
\end{equation}
If in addition (H3) holds, then $t \mapsto H(\nu _{(t)})$ is constant.
\end{theorem}

\begin{remark}
Existence of $R_0$ in the global time condition (\ref{maxtime}) is guaranteed by the fact that $\nu $ is compactly supported. This is a crucial property for all our applications. 
\end{remark}

\end{document}